\documentclass[10pt]{article}
\usepackage{amstext}
\usepackage{amsfonts}
\usepackage{amssymb}
\usepackage{amsbsy}
\usepackage{amsmath}
\usepackage{latexsym}
\usepackage{xy}
\usepackage{hhline}
\xyoption{all}
\newcommand{\vtx}[1]{*+[o][F-]{\scriptscriptstyle #1}} 

\newcounter{num}[section] %

\newenvironment{theo}
{\refstepcounter{num}%
\bigskip\noindent{\bf Theorem~\arabic{section}.\arabic{num}. }\it}

\newenvironment{lemma}
{\refstepcounter{num}%
\bigskip\noindent{\bf Lemma~\arabic{section}.\arabic{num}. }\it}

\newenvironment{example}
{\refstepcounter{num}%
\bigskip\noindent{\bf Example~\arabic{section}.\arabic{num}.}}

\newenvironment{remark}
{\refstepcounter{num}%
\bigskip\noindent{\bf Remark~\arabic{section}.\arabic{num}.}\it}

\newcommand{\Ref}[1]{(\ref{#1})}

\newcounter{thepic}
\newcommand{\pic}
{\refstepcounter{thepic}%
Figure~\arabic{thepic}.}

\newenvironment{proof}{\medskip\noindent{\it Proof. }}
{$\Box$ \bigskip}
\newenvironment{proof_theo_part3}{\noindent\textbf{Proof of Theorem~\ref{theo_part3}. }}
{$\Box$ \bigskip}

\newenvironment{eq}{\begin{equation}}{\end{equation}}

\newcommand{\si}{\sigma}

\newcommand{\la}{\lambda}
\newcommand{\de}{\delta}

\newcommand{\ch}[1]{\check{#1}}
\newcommand{\ov}[1]{\overline{#1}}
\newcommand{\un}[1]{{\underline{#1}} }

\newcommand{\tr}{\mathop{\rm tr}}

\newcommand{\mdeg}{\mathop{\rm mdeg}}
\newcommand{\degII}[2]{\mathop{\rm {deg}}^{o}_{#1}(#2)}

\newcommand{\Char}{\mathop{\rm char}}
\newcommand{\sign}{\mathop{\rm{sgn }}}

\newcommand{\NN}{{\mathbb{N}} }
\newcommand{\ZZ}{{\mathbb{Z}} }

\newcommand{\Q}{\mathcal{Q}}    
\newcommand{\QG}{\mathcal{G}}    
\newcommand{\QT}{\mathcal{T}}    
\newcommand{\n}{\boldsymbol{n}} 
\newcommand{\Ver}[1]{\mathop{{\rm ver}(#1)}} 
\newcommand{\Arr}[1]{\mathop{{\rm arr}(#1)}} 
\newcommand{\path}{\mathop{\rm path}} 

\newcommand{\OmegaNull}{\Omega_1}  
\newcommand{\OmegaZwei}{\Omega_2}
\newcommand{\OmegaDrei}{\Omega_3}
\newcommand{\OmegaMy}{\Omega}

\newcommand{\Symm}{\mathcal{S}}                 

\newcommand{\rectangle}[2]{
\begin{picture}(0,0)
\put(-#1,-#2){\line(1,0){#1}}\put(0,-#2){\line(1,0){#1}}
\put(-#1,#2){\line(1,0){#1}}\put(0,#2){\line(1,0){#1}}
\put(-#1,-#2){\line(0,1){#2}}\put(-#1,0){\line(0,1){#2}}
\put(#1,-#2){\line(0,1){#2}}\put(#1,0){\line(0,1){#2}}
\end{picture}}

\newcommand{\rombL}{
\begin{picture}(0,0)
\put(80,0){\vector(-4,3){40} \put(13,15){$\scriptstyle 2$}}%
\put(40,30){\vector(0,-1){22} \put(3,-14){$\scriptstyle s$}} %
\put(40,-8){\vector(0,-1){22} \put(3,-13){$\scriptstyle s$}}%
\put(40,-30){\vector(-4,3){40}\put(3,10){$\scriptstyle t$}}%
\put(40,-30){\vector(4,3){40} \put(-17,10){$\scriptstyle 2$}}%
\put(40,4){\circle*{1}}\put(40,0){\circle*{1}}\put(40,-4){\circle*{1}}%
\put(-10,31){
{\xymatrix@C=1.3cm@R=0.9cm{ %
& \ar@/^/@{<-}[ld] \ar@/_/@{<-}[ld] \\%
\\}}}%
\put(19,16){\put(2,-2){\circle*{1}}\put(0,0){\circle*{1}}\put(-2,2){\circle*{1}}}%
\put(11,22){$\scriptstyle 1$}\put(25,4){$\scriptstyle 1$}\put(-6,-2){$\scriptstyle u$}\put(39,32){$\scriptstyle v$}%
\end{picture}} %

\newcommand{\rombC}{
\begin{picture}(0,0)
\put(0,0){\vector(4,3){40} \put(-27,15){$\scriptstyle 2$}}%
\put(80,0){\vector(-4,3){40} \put(13,15){$\scriptstyle 2$}}%
\put(40,30){\vector(0,-1){22} \put(3,-14){$\scriptstyle 4$}}%
\put(40,-8){\vector(0,-1){22} \put(3,-13){$\scriptstyle 4$}}%
\put(40,-30){\vector(-4,3){40}\put(3,10){$\scriptstyle 2$}}%
\put(40,-30){\vector(4,3){40} \put(-17,10){$\scriptstyle 2$}}%
\put(40,4){\circle*{1}}\put(40,0){\circle*{1}}\put(40,-4){\circle*{1}}%
\end{picture}} %

\newcommand{\rombR}{
\begin{picture}(0,0)
\put(0,0){\vector(4,3){40} \put(-27,15){$\scriptstyle 1$}}%
\put(40,30){\vector(0,-1){22} \put(3,-14){$\scriptstyle 1$}}%
\put(40,-8){\vector(0,-1){22} \put(3,-13){$\scriptstyle 1$}}%
\put(40,-30){\vector(-4,3){40}\put(3,10){$\scriptstyle 1$}}%
\put(40,4){\circle*{1}}\put(40,0){\circle*{1}}\put(40,-4){\circle*{1}}%
\end{picture}} %

\newcommand{\rombCnull}{
\begin{picture}(0,0)
\put(0,0){\vector(4,3){40} \put(-27,15){$\scriptstyle 2$}}%
\put(80,0){\vector(-4,3){40} \put(13,15){$\scriptstyle 1$}}%
\put(40,30){\vector(0,-1){22} \put(3,-14){$\scriptstyle 3$}}%
\put(40,-8){\vector(0,-1){22} \put(3,-13){$\scriptstyle 3$}}%
\put(40,-30){\vector(-4,3){40}\put(3,10){$\scriptstyle 2$}}%
\put(40,-30){\vector(4,3){40} \put(-17,10){$\scriptstyle 1$}}%
\put(40,4){\circle*{1}}\put(40,0){\circle*{1}}\put(40,-4){\circle*{1}}%
\end{picture}} %

\newcommand{\rombRnull}{
\begin{picture}(0,0)
\put(0,0){\vector(4,3){40} \put(-27,15){$\scriptstyle 2$}}%
\put(40,30){\vector(0,-1){22} \put(3,-14){$\scriptstyle 2$}}%
\put(40,-8){\vector(0,-1){22} \put(3,-13){$\scriptstyle 2$}}%
\put(40,-30){\vector(-4,3){40}\put(3,10){$\scriptstyle 2$}}%
\put(40,4){\circle*{1}}\put(40,0){\circle*{1}}\put(40,-4){\circle*{1}}%
\end{picture}} %

\newcommand{\rombSmall}{
\begin{picture}(0,0)
\put(0,0){\vector(1,1){30}}%
\put(60,0){\vector(-1,1){30}}%
\put(30,30){\vector(0,-1){22}}%
\put(30,-8){\vector(0,-1){22}}%
\put(30,-30){\vector(-1,1){30}}%
\put(30,-30){\vector(1,1){30}}%
\put(30,4){\circle*{1}}\put(30,0){\circle*{1}}\put(30,-4){\circle*{1}}%
\end{picture}} %

\newcommand{\rombSmallC}{
\begin{picture}(0,0)
\put(0,0){\vector(2,1){60}}%
\put(120,0){\vector(-2,1){60}}%
\put(60,30){\vector(0,-1){22}}%
\put(60,-8){\vector(0,-1){22}}%
\put(60,-30){\vector(-2,1){20}}\put(20,-10){\vector(-2,1){20}}%
\put(30,-15){\put(0,0){\circle*{1}}\put(4,-2){\circle*{1}}\put(-4,2){\circle*{1}}}%
\put(60,-30){\vector(2,1){20}}\put(100,-10){\vector(2,1){20}}%
\put(90,-15){\put(0,0){\circle*{1}}\put(4,2){\circle*{1}}\put(-4,-2){\circle*{1}}}%
\put(60,4){\circle*{1}}\put(60,0){\circle*{1}}\put(60,-4){\circle*{1}}%
\put(59,32){$\scriptstyle v_1$}%
\put(63,8){$\scriptstyle v_2$}%
\put(63,-12){$\scriptstyle v_{t-1}$}%
\put(59,-36){$\scriptstyle v_t$}%
\end{picture}} %

\newcommand{\cycleSmall}{
\begin{picture}(0,0)
\put(-5,0){\xymatrix@C=1.75cm@R=2cm{ %
\ar@/^/@{->}[r]&\\
\\
}}%
\put(50,0){\vector(-1,-2){10}}%
\put(10,-20){\vector(-1,2){10}}%
\put(25,-20){\put(0,0){\circle*{1}}\put(4,0){\circle*{1}}\put(-4,0){\circle*{1}}}%
\end{picture}} %

\begin{document}
\title{Indecomposable invariants of quivers for dimension $(2,\ldots,2)$ and maximal paths, II}
 \author{
A.A. Lopatin \\
Omsk Branch of Institute of Mathematics, SBRAS, \\
Pevtsova street, 13,\\
Omsk 644099 Russia \\
artem\underline{ }lopatin@yahoo.com \\
http://www.iitam.omsk.net.ru/\~{}lopatin/\\
}
\date{} 
\maketitle

\begin{abstract}
An upper bound on degrees of elements of a minimal generating system for invariants of
quivers of dimension $(2,\ldots,2)$ is established over a field of arbitrary
characteristic and its precision is estimated. The proof is based on the reduction to the
problem of description of maximal paths satisfying certain condition.
\end{abstract}




\section{Introduction}\label{section_intro}
We work over an infinite field $K$ of arbitrary characteristic $\Char(K)$. All vector
spaces, algebras, and modules are over $K$ unless otherwise stated and all algebras are
associative.

This paper is a completion of~\cite{Lopatin_Comm3} and we use the same notations as in~\cite{Lopatin_Comm3}. Let us recall some of them. A {\it quiver} $\Q=(\Ver{\Q},\Arr{\Q})$ is a finite oriented graph, where $\Ver{\Q}$ is
the set of vertices and $\Arr{\Q}$ is the set of arrows. The notion of quiver was introduced by Gabriel in~\cite{Gabriel_1972} as an
effective mean for description of different problems of the linear algebra. 

The head (the tail, respectively) of an arrow $a$ is denoted by $a'$ ($a''$, respectively). We say that  $a=a_1\cdots a_s$ is a {\it path} in $\Q$ (where
$a_1,\ldots,a_s\in\Arr{\Q}$), if $a_1'=a_2'',\ldots, a_{s-1}'=a_s''$; and $a$ is a {\it
closed} path in a vertex $v$, if $a$ is a path and $a_1''=a_s'=v$. The head of the path
$a$ is $a'=a_s'$ and the tail of $a$ is $a''=a_1''$. Denote
$\Ver{a}=\{a_1'',a_1',\ldots,a_s'\}$, $\Arr{a}=\{a_1,\ldots,a_s\}$, and $\deg(a)=s$. Given 
a closed path $a$ and $w\in\Ver{\Q}$, we set $\deg_w(a)=\#\{i\,|\,a_i'=w,\,1\leq i\leq s\}$. A closed path $a$ is called {\it
primitive} if $\deg_w(a)=1$ for all $w\in\Ver{a}$. Denote by $m(\Q)$ the maximal degree of primitive closed paths in
$\Q$. Closed paths $a_1,\ldots,a_s$ in $\Q$ are called {\it incident} if
$a_1'=\cdots=a_s'$.

For a quiver $\Q$ and a {\it dimension vector} $\n=(\n_{v}\,|\,v\in\Ver{\Q})$ denote by
$I(\Q,\n)$ the {\it algebra of invariants} of representations of $\Q$.  The algebra
$I(\Q,\n)$ is embedded into the algebra of (commutative) polynomials
$K[x_{ij}(a)\,|\,a\in\Arr{\Q},\ 1\leq i\leq \n_{a'},\,1\leq j\leq \n_{a''}]$. Denote by %
$X_{a}=(x_{ij}(a))$ the $n_{a'}\times n_{a''}$ {\it generic} matrix and by $\si_k(X)$ the $k$-th coefficient
in the characteristic polynomial of an $n\times n$ matrix $X$,
i.e., %
$$\det(\la E-X)=\la^n-\si_1(X)\la^{n-1}+\cdots+(-1)^n\si_n(X).
$$%

\begin{theo}\label{theo_Donkin}(Donkin~\cite{Donkin_1994}) The $K$-algebra $I(\Q,\n)$ is generated by $\si_k(X_{a_s}\cdots X_{a_1})$ for all closed paths $a=a_1\cdots a_s$ in $\Q$ (where
$a_1,\ldots,a_s\in\Arr{\Q}$) and $1\leq k\leq n_{a'}$. 
\end{theo}

\bigskip
\noindent Notice that $I(\Q,\n)$ has a
grading by degrees that is given by the formula: $\deg(\si_k(X_{a_s}\cdots X_{a_1}))=ks$.

Investigation of $I(\Q,\n)$ was originated from the partial case of a quiver with one vertex. Sibirskii~\cite{Sibirskii_1968}, Razmyslov~\cite{Razmyslov_1974} and Procesi~\cite{Procesi_1976} described generators and relations in the case of characteristic zero field. As about the case of arbitrary characteristic, the first step was performed by Donkin in~\cite{Donkin_1992a}, where he established generators. Relations between generators of $I(\Q,\n)$ were established by Domokos~\cite{Domokos_1998} in characteristic zero case and by Zubkov~\cite{Zubkov_Fund_Math_2001} in arbitrary characteristic case. Theorem~\ref{theo_Donkin} was generalized to the case of action of arbitrary classical linear groups in~\cite{Lopatin_so_inv} using approach from~\cite{LZ1}.

By the Hilbert--Nagata Theorem on invariants, $I(\Q,\n)$ is a finitely generated
graded algebra. But the mentioned generating system is not finite. So it gives rise to
the problem to find out a minimal (by inclusion) homogeneous system of generators
(m.h.s.g.). Let $D(\Q,\n)$ be the 
least upper bound for the degrees of elements of a m.h.s.g.~of $I(\Q,\n)$. Note that taking elements from Theorem~\ref{theo_Donkin} of the degree less or equal to $D(\Q,\n)$ we obtain the finite system of generators. A {\it decomposable} invariant is equal to a polynomial in elements of strictly lower degree. Obviously, $D(\Q,\n)$ is equal to
the highest degree of indecomposable invariants. 

In~\cite{Lopatin_Comm3} we established an upper bound on $D(\Q,\n)$ for an arbitrary quiver $\Q$ and $\n=(2,2,\ldots,2)$. In this paper we improve essentially the mentioned upper bound and estimate its precision (see Theorem~\ref{theo_main_invariants} and Remark~\ref{remark_compare}). Note that for a quiver with one vertex and $\n=(2)$ a m.h.s.g.~was found in~\cite{Sibirskii_1968},~\cite{Procesi_1984},~\cite{DKZ_2002}; in case $\n=(3)$ a m.h.s.g.~was described in~\cite{Lopatin_Comm1},~\cite{Lopatin_Comm2} and a system of parameters for a quiver with three loops was found in~\cite{Lopatin_Sib}. A m.h.s.g.~for the algebra of semi-invariants of a quiver of dimension $(2,\ldots,2)$ was established in~\cite{Lopatin_semi2}. References to other results on generating systems for invariants are given, for example,  in~\cite{Lopatin_Comm3}.

Without loss of generality we can assume that $\Q$ is a {\it strongly connected} quiver, i.e., there exists a closed path in $\Q$ that contains all vertices of $\Q$ (for the details, see Section~1 of~\cite{Lopatin_Comm3}).

For positive integers $n,d,m$ define $M(n,d,m)$ as follows:
\begin{enumerate}
\item[1)] if $\Char(K)=2$, then %
$$
M(n,d,m)=\left\{ %
\begin{array}{rl}
2m,& \text{if } d=n=m\\
2m(d-n+\frac{1}{2}),& \text{if } d< n+2\left[\frac{n-1}{m}\right]\text{ and } n>m\geq2\\
m(d-n-1)+2n,& \text{otherwise}\\
\end{array} \right.;
$$
\item[2)] if $\Char(K)\neq2$, then
$$
M(n,d,m)=\left\{ %
\begin{array}{rl}
2n,& \text{if } n=m\text{ and }d\in\{n,n+1\}\\
3n,& \text{otherwise}\\
\end{array} \right..
$$
\end{enumerate}
Here $[\alpha]$ stands for the greatest integer that does not exceed
$\alpha$. 

Denote by $\Q(n,d,m)$ the set of all strongly connected quivers $\Q$ with
$\#\Ver{\Q}=n$, $\#\Arr{\Q}=d$, and $m(\Q)=m$. A criterion when $\Q(n,d,m)$ is
not empty is given by Lemma~\ref{lemma_criterion_for_QuiverNull}. For short, we write $D(n,d,m)$ for $\max\{D(\Q,(2,\ldots,2))\,|\,\Q\in\Q(n,d,m)\}$. Our main result is the
following theorem.

\begin{theo}\label{theo_main_invariants}
For $\Q(n,d,m)\neq\emptyset$ we have $D(n,d,m)\leq M(n,d,m)$. Moreover,
\begin{enumerate}
\item[1)] if $\Char(K)=2$, then 
$$D(n,d,m)\geq M(n,d,m)-m.$$

\item[2)] if $\Char(K)\neq2$, $d\geq n+2\left[\frac{n-1}{m}\right]+m$ or $n=m$, then $$D(n,d,m)= M(n,d,m).$$
\end{enumerate}
\end{theo}

\medskip
As immediate corollary of this theorem we obtain that if  $\Q\in\Q(n,d,m)$, then the algebra of invariants $I(\Q,(\de_1,\ldots,\de_n))$ with $\de_1,\ldots,\de_n\leq 2$ is generated by elements of degree at most $M(n,d,m)$.

\begin{remark}\label{remark_compare} Let $\Char(K)=2$. In~\cite{Lopatin_Comm3} we gave the following upper bound: $D(n,d,m)\leq md$ for $\Q(n,d,m)\neq\emptyset$. By Theorem~\ref{theo_main_invariants}, for $m>2$ the deviation of this upper bound is 
\begin{eq}\label{eq_estimation}
md - D(n,d,m)\to \infty\text{ as } n,d\to\infty,
\end{eq} 
where we assume that $m$ is fixed and $n,d\to\infty$ in such a way that at each step $\Q(n,d,m)\neq\emptyset$. But the deviation of the upper bound from Theorem~\ref{theo_main_invariants} is less or equal to the constant $m$, i.e., 
$$0\leq M(n,d,m)-D(n,d,m)\leq m.$$
\end{remark}

\bigskip
As in~\cite{Lopatin_Comm3}, for a quiver $\Q$ introduce an equivalence $\equiv$ on the set of all closed paths
extended with an additional symbol $0$. For any paths $a,b$ such that $ab$ is a closed path and any incident closed paths $a_1,a_2,\ldots$ we
define
\begin{enumerate}
\item[1.] $ab\equiv ba$;

\item[2.] $a_{\si(1)}\cdots a_{\si(t)}\equiv \sign(\si)\,a_1\cdots a_t$, where $t\geq2$ and $\si\in
\Symm_t$;

\item[3.] $a_1^2a_2\equiv0$;

\item[4.] if $\Char(K)=2$, then $a_1^2\equiv0$; if $\Char(K)\neq2$, then
$a_1a_2a_3a_4\equiv0$.
\end{enumerate}
We write $M(\Q)$ for the maximal degree of a closed path $a$ in $\Q$ satisfying
$a\not\equiv0$. The following lemma is Lemma~1.2 of~\cite{Lopatin_Comm3}, which was proved using~\cite{Zubkov_Fund_Math_2001}.

\begin{lemma}\label{lemma_reduction}
Let $a=a_1\cdots a_s$ be a closed path in $\Q$, where $a_1,\ldots,a_s\in\Arr{\Q}$. Then
$\tr(X_{a_s}\cdots X_{a_1})\in I(\Q,(2,2,\ldots,2))$ is decomposable if and only if
$a\equiv0$.
\end{lemma}

\begin{remark}\label{remark_det}
Let $a=a_1\cdots a_s$ be a closed path in $\Q$, where $a_1,\ldots,a_s\in\Arr{\Q}$. If
$q=\det(X_{a_s}\cdots X_{a_1})\in I(\Q,(2,\ldots,2))$ is indecomposable, then $a$ is a primitive closed path and $\deg(a)\leq m$. Thus, $\deg(q)\leq
M(n,d,m)$.
\end{remark}

\bigskip

Section~\ref{section_auxiliary} contains necessary definitions and results from~\cite{Lopatin_Comm3}. If $\Char(K)\neq2$, then the upper bound on $M(\Q)$ is
calculated in Lemma~\ref{lemma_char_0}; otherwise, we establish the upper bound on
$M(\Q)$ in Theorems~\ref{theo_part3} and~\ref{theo_part6}. In Lemma~\ref{lemma_example}
we estimate a precision of the given upper bound. Taking into account
Lemma~\ref{lemma_reduction} and Remark~\ref{remark_det} together with the fact that
$I(\Q,(2,2,\ldots,2))$ is generated by indecomposable invariants, we complete the
proof of Theorem~\ref{theo_main_invariants}.

In Sections~\ref{section_part4}--\ref{section_part6} we assume that $\Char(K)=2$.
Sections~\ref{section_part4},~\ref{section_part2}, and~\ref{section_part3} are dedicated
to the proof of Theorem~\ref{theo_part3}, which consists of two steps.

At first, we introduce the set of multidegrees $\OmegaZwei(\Q)$ with the property that if
$h$ is a closed path and $\mdeg(h)\in\OmegaZwei(\Q)$, then $h\not\equiv0$ (see
Section~\ref{section_part4} and Remark~\ref{remark_inclusions}). Moreover,
Lemma~\ref{lemma_mdeg} implies that $\OmegaZwei(\Q)$ is the maximal (by inclusion) set
with the given property. In Theorem~\ref{theo_part4} of Section~\ref{section_part4} we
give some upper bound on $|\un{\de}|$ for $\un{\de}\in\OmegaZwei(\Q)$. Note that there
can be a closed path $h\not\equiv0$ such that $\mdeg(h)\not\in\OmegaZwei(\Q)$ (see
Example~\ref{ex_proper}).

During the second step we extract some information from the fact that $h\not\equiv0$ (see
Lemma~\ref{lemma_partII_main2}). Then we find out a closed subpath $c$ in $h$ such that
for two arrows $b_1,b_2$ of $c$ we have $\deg_{b_1}(h)=\deg_{b_2}(h)=1$ and some
additional properties are valid (see Lemma~\ref{lemma_p26}).  The main idea of the proof
of Theorem~\ref{theo_part3} is to substitute $c$ with a loop in order to obtain a quiver
$\QG$ with $\#\Arr{\QG}<\#\Arr{\Q}$ and to use induction hypothesis. The main difficulty
is that we can not claim that $c$ is a primitive closed path, thus we can not say that
$\deg(c)\leq m$. To estimate $\deg(c)$ we apply Lemma~\ref{lemma_p28}.

Section~\ref{section_part6} contains the proof of Theorem~\ref{theo_part6}. In
Section~\ref{section_example} we consider some examples in order to prove
Lemma~\ref{lemma_example}.

\section{Auxiliary results}\label{section_auxiliary}

\subsection{Notations}\label{subsection_notations}
For a path $a=a_1\cdots a_s$ in a quiver $\Q$, where $a_1,\ldots,a_s\in\Arr{\Q}$, and $b\in\Arr{\Q}$, $v\in\Ver{\Q}$, we set  
\begin{enumerate}
\item[$\bullet$] $\deg_b(a)=\#\{i\,|\,a_i=b,\,1\leq i\leq s\}$; 

\item[$\bullet$] $\deg_v(a)=\max\{m_1,m_2\}$, where $m_1=\#\{i\,|\,a_i'=v,\,1\leq i\leq s\}$ and $m_2=\#\{i\,|\,a_i''=v,\,1\leq i\leq s\}$;

\item[$\bullet$] $\degII{v}{a}=\#\{i\,|\,a_i'=v,\,1\leq i\leq s-1\}$.
\end{enumerate}
Let $\un{\de}\in\NN^{\#\Arr{\Q}}$, where $\NN$ stands for non-negative
integers. Then the path $a$ is called {\it $\un{\de}$-double} if $a$ is a primitive closed path and $\de_{a_i}\geq2$ for all $i$. The definition of {\it strongly connected components} of an arbitrary quiver $\QG$ is well known (for example, see Section~1 of~\cite{Lopatin_Comm3}). The following notions were defined in~Section~5 of~\cite{Lopatin_Comm3}:
\begin{enumerate}
\item[$\bullet$] the {\it multidegree} $\mdeg(a)$ of a path $a$;

\item[$\bullet$] the {\it empty path} $1_v$ in a vertex $v$;

\item[$\bullet$] a {\it subpath} of a path $a$;

\item[$\bullet$] {\it $h$-restriction} of $\Q$ to $V$, where $V\subset\Ver{\Q}$ and $h$ is a path in $\Q$ (see also Example~5.1 of~\cite{Lopatin_Comm3}).
\end{enumerate}

Denote by $\path(\Q)$ the set of all paths and
empty paths in $\Q$. If we consider a path, then we assume that it is non-empty unless
otherwise stated; if we write $a\in\path(\Q)$, then we assume that a path $a$ can be
empty. 

Dealing with equivalences we use the following conventions. If we write $a\equiv b$,
then we assume that $a$ and $b$ are closed paths in $\Q$. If we write $ab$ for paths $a$
and $b$, then we assume that $a'=b''$. To explain how we apply formulas to prove some
equivalence $a\equiv b$ we split the word $a$ into parts using dots. 

For closed paths $a,b$ we write $a\sim b$ if $a=c_1c_2$ and $b=c_2c_1$ for some
$c_1,c_2\in\path(\Q)$. For $\un{\de},\un{\theta}\in\NN^l$ we
set $\un{\de}\geq\un{\theta}$ if and only if $\de_i\geq \theta_i$ for all $i$ and
define $|\un{\de}|=\de_1+\cdots+\de_l$.

Let $x_1,\ldots,x_s$ be all arrows in $\Q$ from $u$ to $v$, where $u,v\in\Ver{\Q}$. Then
denote by $\ch{x}$ any arrow from $x_1,\ldots,x_s$, by $\{\ch{x}\}$ the set
$\{x_1,\ldots,x_s\}$, and say that $\ch{x}$ is an arrow from $u$ to $v$. Schematically,
we depict arrows $x_1,\ldots,x_s$ as
$$
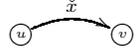
\vcenter{
\xymatrix@C=1cm@R=1cm{ %
\vtx{u}\ar@/^/@{->}[r]^{\ch{x}} &\vtx{v}\\
}}.
$$
For a path $a$ in $\Q$ denote $\deg_{\ch{x}}(a)=\sum_{i=1}^s\deg_{x_i}(a)$. As an
example, an expression $\ch{x}a_1\cdots\ch{x}a_k$ stands for a path $x_{i_1}a_1\cdots
x_{i_k}a_k$ for some $1\leq i_j\leq s$ ($1\leq j\leq k$). Similarly, if $x_1,\ldots,x_s$
are loops in $v\in\Ver{\Q}$, then $\ch{x}^k$ stands for a closed path $x_{i_1}\cdots
x_{i_k}$ for some $i_1,\ldots,i_k$.

The next two lemmas are well known.

\begin{lemma}\label{lemma_mdeg}
Suppose $\Q$ is a strongly connected quiver and $\un{\de}\in\NN^{\#\Arr{\Q}}$. Then the
following conditions are equivalent:
\begin{enumerate}
\item[a)] There is a closed path $h$ in $\Q$ such that $\mdeg(h)=\un{\de}$ and
$\Arr{h}=\Arr{\Q}$; in particular, $\Ver{h}=\Ver{\Q}$.

\item[b)] We have $\de_a\geq1$ for all $a\in\Arr{\Q}$ and $\sum_{a'=v}\de_a =
\sum_{a''=v}\de_a$ for all $v\in\Ver{\Q}$, where the sums range over all $a\in\Arr{\Q}$
satisfying the given conditions.
\end{enumerate}
\end{lemma}

\bigskip
\noindent We write $\de(i,j)$ for the Kronecker symbol. 

\begin{lemma}\label{lemma_criterion_for_QuiverNull}
For positive integers $n,d,m$ the set $\Q(n,d,m)$ is not empty if and only if
one of the following possibilities holds:
\begin{enumerate}
\item[a)] $n=m=1$;

\item[b)] $n\geq m\geq2$ and $d\geq n+l-\de(0,r)$, where $n-1=l(m-1)+r$, $l\geq1$, and
$0\leq r\leq m-2$.
\end{enumerate}
\end{lemma}

\begin{lemma}\label{lemma_subquiver}
Suppose $\Q_1,\Q_2$ are strongly connected quivers and $\Q_1\subset\Q_2$. Then
$$\#\Arr{\Q_2}-\#\Arr{\Q_1}\geq \#\Ver{\Q_2}-\#\Ver{\Q_1}+1.$$
\end{lemma} %
\begin{proof}
For every $v\in\Ver{\Q_2}\backslash\Ver{\Q_1}$ there is an
$a\in\Arr{\Q_2}\backslash\Arr{\Q_1}$ with $a'=v$. There also exists a
$b\in\Arr{\Q_2}\backslash\Arr{\Q_1}$ satisfying $b'\in\Ver{\Q_1}$. These remarks imply
the required formula.
\end{proof}

\subsection{Basic equivalences}\label{subsection_basic}

\begin{lemma}\label{lemma_char_0}
Suppose $\Char(K)\neq2$. If $\Q\in\Q(n,d,m)$, $h$ is a closed path in $\Q$, and
$h\not\equiv0$, then $\deg(h)\leq M(n,d,m)$.
\end{lemma}
\begin{proof} We claim that $\deg(h)\leq 3n$. If $\deg(h)>3n$, then there is a vertex $v\in\Ver{\Q}$ such that $\deg_v(h)\geq 4$. Therefore, $h\equiv h_1\cdots h_4$ for some closed paths $h_1,\ldots,h_4$ in $v$. Thus $h\equiv0$ by the definition of the equivalence $\equiv$; a contradiction.

To complete the proof, it is enough to consider the case of $n=m$ and $d\in\{n,n+1\}$.

1. If $d=n$, then $\Arr{\Q}=\{a_1,\ldots,a_n\}$, where $a=a_1\cdots a_n$ is a primitive closed path. Then $h\equiv a^s$ for some $s>0$. If $s\geq3$, then $h\equiv0$; a contradiction. Thus $\deg(h)\leq 2n$. The case of $n=1$ and $d=n+1$ can be treated similarly. 

2. Let $n=m\geq2$ and $d=n+1$. In this case $\Q$ is
$$
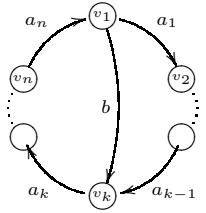
\vcenter{
\xymatrix@=.5cm{ %
&\vtx{v_1}  \ar@/_/@{<-}[ld]_{a_n}  \ar@/^/@{->}[rd]^{a_1}  \ar@/^/@{->}[ddd]_{b}& \\
\vtx{v_n}  \ar@/_/@{..}[d]  &&\vtx{v_2} \ar@/^/@{..}[d]\\
\vtx{\quad\,}  \ar@/_/@{<-}[rd]_{a_{k}}  &&\vtx{\quad\,} \ar@/^/@{->}[ld]^{a_{k-1}}\\
&\vtx{v_k}    & \\
}},
$$
where $1\leq k\leq n$. Denote $a=a_1\ldots a_n$ and 
$$
c=\left\{
\begin{array}{rl}
b, & k=1\\
ba_k\ldots a_n, & \text{otherwise}\\
\end{array}
\right..$$%
We have $h\equiv a^r c^s$ for some $r,s\geq0$. If $r=0$ or $s=0$, then $\deg(h)\leq 2n$ (see Part~1 of the lemma). Assume that $r,s>0$. If $r\geq2$ or $s\geq2$, then $h\equiv0$; a contradiction. Hence $\deg(h)=n+\deg{c}\leq 2n$.
\end{proof}

In what follows we assume that $\Char(K)=2$ unless otherwise stated. We will use the following remark without references to it.

\begin{remark}\label{remark_no_change}
Suppose $f,h$ are closed paths in $\Q$ and $b$ is a subpath of $f$. Let the equivalence
$f\equiv h$ follows from the formulas of the form $a_{\si(1)}\cdots a_{\si(t)}\equiv
a_1\cdots a_t$, where $a_1,\ldots,a_t$ are closed paths in $v\in\Ver{\Q}$ satisfying
$\degII{v}{b}=0$, $t\geq2$, and $\si\in \Symm_t$. Then $b$ is also a subpath of $h$.
\end{remark}

\bigskip
There following three lemmas are Lemmas~6.3,~6.8, and~6.9 of~\cite{Lopatin_Comm3}, respectively. 

\begin{lemma}\label{lemma_L0}
Let $h$ be a closed path in $\Q$ and $\{\ch{p}\}$ be loops of $\Q$ in some
$v\in\Ver{\Q}$. Then $h\equiv \ch{p}^kb$, where $k\geq0$, $b\in\path(\Q)$, and
$\deg_{\ch{p}}(b)=0$.

Moreover, suppose $a\in\Arr{h}$ and $a'\neq a''$. If $a'=v$, then $h\equiv a
\ch{p}^kb_0$; if $a''=v$, then $h\equiv \ch{p}^kab_0$, where, as above,
$\deg_{\ch{p}}(b_0)=0$.
\end{lemma}

\bigskip
Suppose a quiver $\Q$ contains a path $a=a_1\cdots a_s$, where
$a_1,\ldots,a_s\in\Arr{\Q}$ are pairwise different.  Let $h$ be a closed path in $\Q$ such that $\deg_{a_i}(h)\geq2$ for all $i$
and there is a $b\in\Arr{h}$ satisfying $b\neq a_i$ for all $i$.

\begin{lemma}\label{lemma_aaa}
Using the preceding notation we have $h\equiv a_1\cdots a_sf$ for some $f\in\path(\Q)$.
Moreover,
\begin{enumerate}
\item[a)] if $b'=a''_1$, then $h\equiv b\,a_1\cdots a_sf$ for some $f\in\path(\Q)$;

\item[b)] if $b''=a'_s$, then $h\equiv a_1\cdots a_s\,bf$ for some $f\in\path(\Q)$.
\end{enumerate}
\end{lemma}

\bigskip
Let $a$ and $h$ be paths as above. For $1\leq i\leq s$ denote
$v_i=a_i''$. We assume that the path $a$ is closed and primitive, $s\geq2$, $b'\neq b''$,
and $b',b''\in\{v_2,v_k\}$ for some $k\in\{1,3,4,\ldots,s\}$. Schematically this is
depicted as
$$
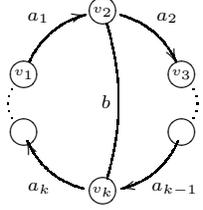
\vcenter{
\xymatrix@=.5cm{ %
&\vtx{v_2}  \ar@/_/@{<-}[ld]_{a_1}  \ar@/^/@{->}[rd]^{a_2}  \ar@/^/@{-}[ddd]_{b}& \\
\vtx{v_1}  \ar@/_/@{..}[d]  &&\vtx{v_3} \ar@/^/@{..}[d]\\
\vtx{\quad\,}  \ar@/_/@{<-}[rd]_{a_{k}}  &&\vtx{\quad\,} \ar@/^/@{->}[ld]^{a_{k-1}}\\
&\vtx{v_k}    & \\
}}.
$$

\begin{lemma}\label{lemma_L4}
Using the preceding notation we have $h\equiv a_1a_2 f_1\, a_1a_2 f_2$ for some
$f_1,f_2\in\path(\Q)$.
\end{lemma}

\begin{lemma}\label{lemma_pA291}
Suppose $\Q$ is a quiver with $n$ vertices and $d$ arrows. Let $h$ be a closed path in $\Q$ and $h\not\equiv0$.
Then there exist pairwise different primitive closed paths $b_1,\ldots,b_r$,
$c_1,\ldots,c_t$ in $\Q$, where $r,t\geq0$ and $r+t\leq d-n+1$, such that
$$\mdeg(h)=\sum_{i=1}^r \mdeg(b_i)+2\sum_{k=1}^t\mdeg(c_k);$$
and there are pairwise different arrows $x_1,\ldots,x_r$, $y_1,\ldots,y_t$,
$z_1,\ldots,z_t$ in $\Q$ satisfying
\begin{eq}\label{eq_pA291_1}
y_j,z_j\in\Arr{c_j}\text{ and }\deg_{y_j}(h)=\deg_{z_j}(h)=2,
\end{eq}
\vspace{-0.5cm}
\begin{eq}\label{eq_pA291_new0}
x_i\in\Arr{b_i}\text{ and }\deg_{x_i}(h)-2\sum_{k=1}^t \deg_{x_i}(c_k)=1 
\end{eq}
for any $1\leq i\leq r$, $1\leq j\leq t$.
\end{lemma}
\begin{proof} The statement of the lemma but the inequality $r+t\leq d-n+1$ follows from Lemma~6.10~\cite{Lopatin_Comm3}. Applying
Lemma~\ref{lemma_subquiver}, we can assume that $\Q=\Q_{\mdeg{h}}$.

Denote by $\QG$ the quiver that is the union of closed paths $b_1,\ldots,b_r$, i.e.,
$\Ver{\QG}=\Ver{b_1}\cup\cdots\cup\Ver{b_r}$ and
$\Arr{\QG}=\Arr{b_1}\cup\cdots\cup\Arr{b_r}$. Let $\QG_1,\ldots,\QG_l$ be the strongly 
connected components of $\QG$. We have $\Arr{\QG_k}=\bigcup_{i\in I_k}\Arr{b_i}$ for some
$I_k\subset [1,r]$  and denote $\#I_k=r_k$ ($1\leq k\leq l$).

We assume that $k=1$. Consider an $i_1\in I_1$ and let $\Q_1$ be the quiver such that
$\Ver{\Q_1}=\Ver{b_{i_1}}$ and $\Arr{\Q_1}=\Arr{b_{i_1}}$. If $\#I_1>1$, then there is an
$i_2\in I_1\backslash\{i_1\}$ satisfying $\Ver{b_{i_2}}\cap\Ver{\Q_1}\neq\emptyset$. By
part~a), we have $x\not\in \Arr{\Q_1}$ for some $x\in\Arr{b_{i_2}}$. Hence there is an
$e_2\in\Arr{b_{i_2}}$ such that $e_2\not\in \Arr{\Q_1}$ and $e'_2\in\Ver{\Q_1}$. We add
the closed path $b_{i_2}$ to $\Q_1$ and obtain a new quiver $\Q_2$, i.e.,
$\Ver{\Q_2}=\Ver{\Q_1}\cup \Ver{b_{i_2}}$ and $\Arr{\Q_2}=\Arr{\Q_1}\cup \Arr{b_{i_2}}$.
Then we repeat this procedure for $\Q_2$ and so on. Finally, we obtain
$\Q_1,\Q_2,\ldots,\Q_{r_1}=\QG_1$ and pairwise different arrows $e_2,\ldots,e_{r_1}$ such
that $e_j\in \Arr{\Q_j}\backslash\Arr{\Q_{j-1}}$ and $e'_j\in\Ver{\Q_{j-1}}$ for any
$2\leq j\leq r_1$. Then for the set $V_1=\{e'_2,\ldots,e'_{r_1}\}$ we have
$\#\{a\in\Arr{\QG_1}|\,a'\in V_1\}\geq \#V_1+(r_1-1)$. Since for every
$v\in\Ver{\QG_1}\backslash V_1$ there is at least one arrow $a\in\Arr{\QG_1}$ with
$a'=v$, we have
$$\#\Arr{\QG_1}\geq \#\Ver{\QG_1}+(r_1-1).$$
The similar formula holds for all $k$. It follows that%
\begin{eq}\label{eq_pA293}
\#\Arr{\QG}\geq \#\Ver{\QG}+(r-l).
\end{eq}
\indent For the quiver $\Q_r=\QG$ there is a $j_1\in[1,t]$ satisfying
$\Ver{c_{j_1}}\cap\Ver{\Q_r}\neq\emptyset$. We add $c_{j_1}$ to $\Q_r$ and denote the
resulting quiver by $\Q_{r+1}$. By~\Ref{eq_pA291_1}, there exists a $g_1\in\Arr{c_{j_1}}$
such that $g_1\not\in \Arr{\Q_r}$ and $g'_1\in\Ver{\Q_r}$. Moreover, if the number of
strongly connected components of $\Q_{r+1}$ is less than the number of strongly connected
components of $\Q_r$, then there also exists a $g_2\in\Arr{c_{j_1}}\backslash\{g_1\}$
such that  $g_2\not\in \Arr{\Q_r}$ and $g'_2\in\Ver{\Q_r}$. We repeat this procedure for
$\Q_{r+1}$ and so on. Finally, we obtain quivers $\Q_r,\Q_{r+1},\ldots,\Q_{r+t}=\Q$ and
pairwise different arrows $g_1,\ldots,g_{t+l-1}$ of $\Q$ such that for the set
$V=\{g'_1,\ldots,g'_{t+l-1}\}$ we have
$$\#\{a\in \Arr{\Q}\backslash\Arr{\QG}|\,a'\in V\}\geq \#V\backslash\!\Ver{\QG}+(t+l-1).$$
Therefore
$$\#\Arr{\Q}\backslash\Arr{\QG}\geq \#\Ver{\Q}\backslash\Ver{\QG}+(t+l-1)$$
and~\Ref{eq_pA293} completes the proof.
\end{proof}

\section{Sets of multidegrees}\label{section_part4}  

Suppose $\Q$ is a strongly connected quiver and $\Char(K)=2$. 

The {\it support} of a non-zero vector $\un{\de}\in\NN^{\#\Arr{\Q}}$ with respect to $\Q$ is the subquiver
$\Q_{\un{\de}}$ of $\Q$ such that $\Arr{\Q_{\un{\de}}}=\{a\in\Arr{\Q}\,|\,\de_a\geq1\}$ and
$\Ver{\Q_{\un{\de}}}=\{a',a''\,|\,a\in \Arr{\Q_{\un{\de}}}\}$. The following remark is extensively applied to established indecomposability of invariants.

\begin{remark}\label{remark_inclusions_new} 
Let $h$ be a closed path in $\Q$. If for any $\mdeg(h)$-double path $a$ we have that the support of $\mdeg(h)-2\mdeg(a)$ is not strongly connected (and is not empty), then $h\not\equiv0$. 
\end{remark}
\begin{proof} 
If $h$ satisfies the condition of the lemma and $h\equiv0$, then 
$h\equiv a^2f$ for some paths $a,f$. Thus the support of $\mdeg(h)-2\mdeg(a)=\mdeg(f)$ is strongly connected; a contradiction.
\end{proof}

For a non-zero vector $\un{\de}\in\NN^{\#\Arr{\Q}}$ we say that 
\begin{enumerate}
\item[$\bullet$] $\un{\de}$ is {\it indecomposable} (with respect to $\Q$) if its support is strongly connected; 

\item[$\bullet$] $\un{\de}$ is {\it decomposable} (with respect to $\Q$) if its support is not strongly connected but is the disjoint union of strongly connected quivers. 
\end{enumerate}%
Observe that $\un{\de}$ can be neither decomposable nor indecomposable. We say that $\un{\de}=\un{\de}^{(1)}+\cdots+\un{\de}^{(r)}$ is the {\it decomposition} of
$\un{\de}$ with respect to $\Q$ if $\un{\de}^{(1)},\ldots,\un{\de}^{(r)}\in
\NN^{\#\Arr{\Q}}$ are non-zero vectors and
$\Q_{\un{\de}^{(1)}},\ldots,\Q_{\un{\de}^{(r)}}$ are pairwise different strongly 
connected components of $\Q_{\un{\de}}$. Obviously,  if $\un{\de}$ is indecomposable, then $r=1$; and if $\un{\de}$ is decomposable, then $r\geq2$. Introduce the following sets:

\begin{enumerate}
\item[a)] the set $\OmegaNull(\Q)$ consists of all $\mdeg(h)$, where $h$ ranges over
closed paths in $\Q$ with $\Arr{h}=\Arr{\Q}$;


\item[b)] the set $\OmegaZwei(\Q)$ consists of such $\un{\de}\in\OmegaNull(\Q)$ that for
every $\un{\de}$-double path $a$ in $\Q$ we have $\un{\de}-2\mdeg(a)$ is decomposable
with respect to $\Q$;

\item[c)] the set $\OmegaDrei(\Q)$ consists of such $\un{\de}\in\OmegaNull(\Q)$ that
there is no $\un{\de}$-double path in $\Q$;

\item[d)] the set $\OmegaMy(\Q)$ consists of such $\mdeg(h)\in\OmegaNull(\Q)$ that $h$ is
a closed path in $\Q$ and $h\not\equiv0$.
\end{enumerate}

For every vector $\un{\de}\in\OmegaNull(\Q)$ there exists its decomposition with respect
to $\Q$ that consists of one summand. Moreover, by Lemma~\ref{lemma_mdeg}, for every
$\un{\theta}\in\OmegaNull(\Q)$ with $\un{\de}-\un{\theta}\geq0$ there also exists a
decomposition of $\un{\de}-\un{\theta}$ with respect to $\Q$.

\begin{remark}\label{remark_inclusions} 
We have the following inclusions: $\OmegaDrei(\Q)\subset \OmegaZwei(\Q)\subset
\OmegaMy(\Q)\subset \OmegaNull(\Q)$.
\end{remark}
\begin{proof} The inclusion $\OmegaZwei(\Q)\subset \OmegaMy(\Q)$ follows from Remark~\ref{remark_inclusions_new}. The remaining inclusions
are trivial.
\end{proof}

\begin{example}\label{ex_proper} Let $h_1=czczxyba$, $h_2=czcbyzxa$ be closed paths in the quiver $\Q$
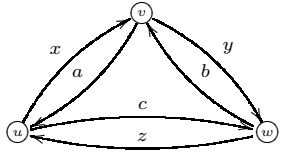
$$ \vcenter{
\xymatrix@C=1.3cm@R=1.3cm{ %
&\vtx{v} %
\ar@/^/@{->}[ld]_{a} \ar@/_/@{<-}[ld]_{x}
\ar@/^/@{->}[rd]^{y} \ar@/_/@{<-}[rd]^{b}& \\
\vtx{u} \ar@/^/@{->}[rr]^{c} \ar@/_/@{<-}[rr]^{z} &&\vtx{w}  }}
$$
Then $h_1\equiv0$, $h_2\not\equiv0$, and
$\mdeg(h_1)=\mdeg(h_2)\in\OmegaMy(\Q)\backslash\OmegaZwei(\Q)$.
\end{example}

\begin{lemma}\label{lemma_part4_main0}
If $\Q\in\Q(n,d,m)$ and $\un{\de}\in\OmegaDrei(\Q)$, then %
$|\un{\de}|\leq m(d-n+1)$.
\end{lemma}
\begin{proof}
By definition, $\un{\de}=\mdeg(h)$ for some closed path $h$ in $\Q$. The definition of
$\OmegaDrei(\Q)$ shows that $h\not\equiv0$. Then Lemma~\ref{lemma_pA291} implies
$\deg(h)\leq m(r+2t)$ and $r+t\leq d-n+1$. Since $t=0$, the proof is completed.
\end{proof}

\bigskip
\noindent{\bf Definition (of a $\un{\de}$-complete chain).}  A {\it chain of paths}
$A=(a_1,\ldots,a_t)$ is an ordered sequence of primitive closed paths satisfying
$\Ver{a_i}\cap \Ver{a_j}=\emptyset$, if $|i-j|>1$; and $\Ver{a_i}\cap
\Ver{a_j}\neq\emptyset$, otherwise. Given $\un{\de}\in\OmegaZwei(\Q)$, the chain of paths
$A$ is called {\it $\un{\de}$-complete} if the following holds.
\begin{enumerate}
\item[1.] The paths $a_1,\ldots,a_t$ are $\un{\de}$-double paths.

\item[2.] For $\un{\theta}=\un{\de}-2\sum_{i=1}^t\mdeg(a_i)$ we have $\un{\theta}\geq0$
and $|\un{\theta}|>0$.

\item[3.] There is a (unique) decomposition
$\un{\theta}=\un{\theta}^{(1)}+\cdots+\un{\theta}^{(r)}$ with respect to $\Q$ and this
decomposition satisfies
\begin{enumerate}
\item[a)] $r\geq2$ and $\un{\theta}^{(i)}\in\OmegaZwei(\Q_{\un{\theta}^{(i)}})$ for all
$i$;

\item[b)] if $t\geq2$, then $r=2$ and %
we have $\Ver{\Q_{\un{\theta}^{(i)}}} \cap \Ver{a_j}\neq\emptyset$ iff $i=j=1$ or
$i=2,\,j=t$.
\end{enumerate}
\end{enumerate}
If there is no $\un{\de}$-double path in $\Q$, then $A=\emptyset$ is called a {\it
$\un{\de}$-complete} chain. Schematically a $\un{\de}$-complete chain $A$ is depicted on
Figure~\ref{pic1} for $t=1$ and on Figure~\ref{pic2} for $t\geq2$, where circles stand
for closed paths and rectangles stand for subquivers of $\Q$:
$$
\begin{picture}(100,110)
\put(-80,70){
\put(0,0){\circle{30}\put(-3,-3){$a_1$}}%
\put(25,20){\rectangle{20}{15}\put(-5,-4){$\Q_{\un{\theta}^{(1)}}$}}%
\put(-20,-25){\rectangle{15}{20}\put(-11,-10){$\Q_{\un{\theta}^{(r)}}$}}%
\put(20,-20){\circle*{2}}%
\put(25,-15){\circle*{2}}%
\put(15,-25){\circle*{2}}%
\put(-10,-65){\pic\label{pic1}}%
}%
\put(130,70){
\put(0,0){\circle*{2}}%
\put(7,0){\circle*{2}}%
\put(-7,0){\circle*{2}}%
\put(-30,0){\circle{30}\put(-3,-3){$a_1$}}%
\put(30,0){\circle{30}\put(-3,-3){$a_t$}}%
\put(-60,0){\rectangle{20}{15}\put(-15,-4){$\Q_{\un{\theta}^{(1)}}$}}%
\put(60,0){\rectangle{20}{15}\put(-5,-4){$\Q_{\un{\theta}^{(2)}}$}}%
\put(-20,-65){\pic\label{pic2}}%
}%
\end{picture}
$$

\begin{lemma}\label{lemma_complete_chain}
For every $\un{\de}\in\OmegaZwei(\Q)$ there exists a $\un{\de}$-complete chain
$A=(a_1,\ldots,a_t)$.
\end{lemma}
\begin{proof} If there is no $\un{\de}$-double path in $\Q$, then
$A=\emptyset$ is a $\un{\de}$-complete chain; otherwise, let $a_1$ be a $\un{\de}$-double
path in $\Q$. Consider the decomposition
$\un{\de}-2\mdeg(a_1)=\un{\de}^{(1)}+\cdots+\un{\de}^{(r)}$ with respect to $\Q$. Since
$\un{\de}\in\OmegaZwei(\Q)$, we have $r\geq2$. If
$\un{\de}^{(i)}\in\OmegaZwei(\Q_{\un{\de}^{(i)}})$ for all $i$, then $A=\{a_1\}$ is a
$\un{\de}$-complete chain. Thus without loss of generality we can assume that
$\un{\de}^{(2)}\not\in\OmegaZwei(\Q_{\un{\de}^{(2)}})$, i.e., there exists a
$\un{\de}^{(2)}$-double path $a_2$ such that $\un{\theta}=\un{\de}^{(2)}-2\mdeg(a_2)$ is
indecomposable. But $\un{\de}-2\mdeg(a_2)$ is decomposable, since
$\un{\de}\in\OmegaZwei(\Q)$. Hence we obtain $\Ver{a_1}\cap \Ver{a_2}\neq\emptyset$ and
$\Ver{a_1}\cap \Ver{\Q_{\un{\theta}}}=\emptyset$ (see the picture).
$$
\begin{picture}(0,100)
\put(0,-15){%
\put(-115,70){
\put(0,0){\circle{30}\put(-3,-3){$a_1$}}%
\put(25,20){\circle{20}\put(-4,-3){$a_2$}}%
\put(-25,20){\rectangle{20}{15}\put(-13,-4){$\Q_{\un{\de}^{(1)}}$}}%
\put(45,20){\rectangle{40}{15}\put(15,-4){$\Q_{\un{\de}^{(2)}}$}}%
\put(-20,-25){\rectangle{15}{20}\put(-11,-10){$\Q_{\un{\de}^{(r)}}$}}%
\put(20,-20){\circle*{2}}%
\put(25,-15){\circle*{2}}%
\put(15,-25){\circle*{2}}%
}%
\put(85,70){
\put(0,0){\circle{30}\put(-3,-3){$a_1$}}%
\put(18,20){\circle{30}\put(-4,-3){$a_2$}}%
\put(-25,20){\rectangle{20}{15}\put(-13,-4){$\Q_{\un{\de}^{(1)}}$}}%
\put(48,20){\rectangle{20}{15}\put(2,-4){$\Q_{\un{\theta}}$}}%
\put(-20,-25){\rectangle{15}{20}\put(-11,-10){$\Q_{\un{\de}^{(r)}}$}}%
\put(20,-20){\circle*{2}}%
\put(25,-15){\circle*{2}}%
\put(15,-25){\circle*{2}}%
}%
\put(0,67){$\Rightarrow$}%
\put(160,65){.}%
}%
\end{picture}
$$

If $r\geq3$, then we consider $a_2$ instead of $a_1$ and obtain that %
$\un{\de}-2\mdeg(a_2)=\un{\theta}'+\un{\theta}$ is the decomposition of
$\un{\de}-2\mdeg(a_2)$, where
$\un{\theta}'=\un{\de}^{(1)}+\un{\de}^{(3)}+\cdots+\un{\de}^{(r)}+2\mdeg(a_1)$ is
indecomposable. Thus without loss of generality we can assume that $r=2$.

We have the decomposition
$\un{\de}-2\mdeg(a_1)-2\mdeg(a_2)=\un{\theta}^{(1)}+\un{\theta}^{(2)}$, where
$\un{\theta}^{(1)}=\un{\de}^{(1)}$ and $\un{\theta}^{(2)}=\un{\theta}$.
Then we consider $\un{\theta}^{(1)}$ and $\un{\theta}^{(2)}$ in the same way as we has
considered $\un{\de}^{(2)}$; and so on. Finally, we obtain a $\un{\de}$-complete chain.
\end{proof}

\bigskip
\noindent{\bf Definition (of a $\un{\de}$-tree).} For $\un{\de}\in\NN^{\#\Arr{\Q}}$ a
triple $(\QT,\un{\de}^{(v)},A_v\,|\,v\in\Ver{\QT})$ is called a {\it $\un{\de}$-tree} if
the following holds:
\begin{enumerate}
\item[1.] $\QT$ is an oriented rooted tree, i.e., there is no closed path in $\QT$, there
is a unique $v_0\in\Ver{\QT}$ with $a'\neq v_0$ for all $a\in\Arr{\QT}$, and for each
other vertex $v$ of $\QT$ there is a unique $a\in\Arr{\QT}$ with $a'= v$. The vertex
$v_0$ is called the {\it root} and a vertex $v\in\Ver{\QT}$ with $v\neq a''$ for all
$a\in\Arr{\QT}$ is called a {\it leaf}.

\item[2.] Suppose $v\in \Ver{\QT}$, then
\begin{enumerate}
\item[a)] $\un{\de}^{(v)}\in\NN^{\#\Arr{\Q}}$ and $\un{\de}^{(v_0)}=\un{\de}$;

\item[b)] $A_v=(a_1,\ldots,a_t)$ is a $\un{\de}^{(v)}$-complete chain;

\item[c)] if $A_v\neq\emptyset$, then
$\un{\de}-2\sum_{i=1}^t\mdeg(a_i)=\un{\de}^{(b_1')}+\cdots+\un{\de}^{(b_r')}$ is the
decomposition with respect to $\Q$, where $b_1,\ldots,b_r$ are all arrows of $\QT$ whose
tails are equal to $v$; otherwise $v$ is a leaf.
\end{enumerate}
\end{enumerate}
In particular, the conditions that $v\in\Ver{\QT}$ is a leaf, $A_v=\emptyset$, and
$\un{\de}^{(v)}\in\OmegaDrei(\Q_{\un{\de}^{(v)}})$ are equivalent. Note that
$\#\Ver{\QT}=1$ iff $\un{\de}\in\OmegaDrei(\Q_{\un{\de}})$. By
Lemma~\ref{lemma_complete_chain}, there exists a $\un{\de}$-tree for every
$\un{\de}\in\OmegaZwei(\Q)$. Observe that for different $u,v\in\Ver{\QT}$ and closed
paths $a\in A_u$, $b\in A_v$ we have $a\neq b$.

\begin{lemma}\label{lemma_pA26_8}
Suppose $\un{\de}\in\OmegaZwei(\Q)\backslash\OmegaDrei(\Q)$ and
$(\QT,\un{\de}^{(v)},A_v\,|\,v\in\Ver{\QT})$ is a {\it $\un{\de}$-tree}. Denote
$l=\#\{v\in\Ver{\QT}\,|\,v\text{ is not a leaf}\,\}$ and define a set $A=\{a\,|\,a\in
A_v\text{ for some }v\in\Ver{\QT}\}$. Then there are pairwise different
$c_1,\ldots,c_{l_1}\in A$ such that $A\backslash \{c_1,\ldots,c_{l_1}\}=B_1\sqcup \cdots
\sqcup B_{l_2}$ is a disjoint union, where $B_1,\ldots,B_{l_2}$ are some chains of paths,
$0\leq l_1<l$, and $1\leq l_2\leq l$.
\end{lemma}
\begin{proof} We assume that $i=1$.
Suppose $v\in\Ver{\QT}$ is not a leaf, $A_v=(a_1,\ldots,a_t)$, and $b_1,\ldots,b_r$ are
arrows of $\QT$ whose tails are equal to $v$. If $t=1$ and there is a $1\leq j\leq r$
such that $A_{b'_j}\neq\emptyset$, then we define $c_i=a_1$, assign $b'_j$ to $c_i$, and
increase $i$ by one.

If $t\geq2$, then $r=2$ by the definition of a complete chain. If we also have
$A_{b'_1}\neq\emptyset$, then we define $c_i=a_1$, assign $b'_1$ to $c_i$, and increase
$i$ by one. If $A_{b'_2}\neq\emptyset$, then we define $c_i=a_t$, assign $b'_2$ to $c_i$,
and increase $i$ by one.

Repeat this procedure for all vertices of $\QT$ that are not leaves and obtain a set of
pairwise different closed paths $C=\{c_1,\ldots,c_{l_1}\}$.
Since we have defined an
injection $C\to\{v\in\Ver{\QT}\,|\,v\text{ is neither a leaf nor the root}\}$, the
inequality $l_1< l$ holds. The claim of the lemma follows from the construction.
\end{proof}

\begin{lemma}\label{lemma_pA26_9}
Suppose $\Q\in\Q(n,d,m)$ and $A=(a_1,\ldots,a_t)$ is a chain of paths such that
for $\un{\de}=2\sum_{i=1}^t\mdeg(a_i)$ we have $\un{\de}\in\OmegaNull(\Q)$ and $t\geq1$.
Then $|\un{\de}|-mt-n_0\leq n$, where
$$
n_0=\left\{%
\begin{array}{rl}
0,& \text{if }t=1\\
\#\Ver{a_1}\cap\Ver{a_2},& \text{if }t=2\\
\#\Ver{a_2}\cup\cdots\cup\Ver{a_{t-1}},& \text{if }t\geq3\\
\end{array}
\right.%
$$
\end{lemma}
\begin{proof}
If $t=1$, then $\deg(a_1)=n$ and $m=n$. Thus $|\un{\de}|-mt-n_0=n$.

If $t\geq2$, then $\frac{1}{2}|\un{\de}|\leq n+n_0$. Therefore %
$|\un{\de}|-mt-n_0=\sum_{i=1}^t(\deg(a_i)-m)+(\frac{1}{2}|\un{\de}|-n_0)\leq n$, since
$\deg(a_i)\leq m$.
\end{proof}

\begin{lemma}\label{lemma_pA26_10}
Suppose $\Q\in\Q(n,d,m)$, $\un{\de}\in\OmegaZwei(\Q)$,
$A=(a_1,\ldots,a_t)\neq\emptyset$ is a $\un{\de}$-complete chain, and
$\un{\theta}=\un{\de}-2\sum_{i=1}^t\mdeg(a_i)$. Let
$\un{\theta}=\un{\theta}^{(1)}+\cdots+\un{\theta}^{(r)}$ be the decomposition with
respect to $\Q$. We define $k=n-\#\Ver{\Q_{\un{\theta}}}$ and assume that
$$|\un{\theta}^{(j)}|\leq m(d_j-n_j)+n_j+\rho_j$$ for any $1\leq j\leq r$, where
$d_j=\#\Arr{\Q_{\un{\theta}^{(j)}}}$, $n_j=\#\Ver{\Q_{\un{\theta}^{(j)}}}$, and
$\rho_j\in\ZZ$. Then
$$|\un{\de}|\leq m(d-n)+n+\sum_{j=1}^r \rho_j+\rho,$$
where $\rho=2\sum_{i=1}^t\deg(a_i)-m(t+1)-k$.
\end{lemma}
\begin{proof}
We define a quiver $\QG$ by $\Ver{\QG}=\Ver{\Q}$ and $\Ver{\QG}=\Ver{\Q_{\theta}}$. Let
$\QG_1,\ldots,\QG_{l}$ be all strongly connected components of $\QG$. Then $l=k+r$ and
for any $1\leq i\leq k+r$ there is an arrow $b$ in
$\Arr{\Q}\backslash\Arr{\Q_{\un{\theta}}}$ such that $b'\in\Ver{\QG_i}$. Moreover, for
any $1\leq i\leq t-1$ there are at least two arrows in
$\Arr{\Q}\backslash\Arr{\Q_{\un{\theta}}}$ whose heads are in
$\Ver{a_i}\cap\Ver{a_{i+1}}$ and every vertex from $\Ver{a_i}\cap\Ver{a_{i+1}}$ is a
strongly connected component of $\QG$. These two remarks imply that
$$d\geq \sum_{j=1}^r d_j+(k+r)+(t-1).$$
Since $r\geq2$, we have $\sum_{j=1}^r d_j \leq d-k-t-1$ and $\sum_{j=1}^r n_j=n-k$.
Clearly,
$$|\un{\de}|\leq m\sum_{j=1}^r d_j +(1-m) \sum_{j=1}^r n_j+\sum_{j=1}^r \rho_j+
2\sum_{i=1}^t\deg(a_i),$$ %
and the above formulas complete the proof.
\end{proof}

\begin{theo}\label{theo_part4}
Suppose $\Q\in\Q(n,d,m)$ is a quiver and $\un{\de}\in\OmegaZwei(\Q)$. Then %
$|\un{\de}|\leq m(d-n-1)+2n$.
\end{theo}
\begin{proof}
If $\un{\de}\in\OmegaDrei(\Q)$, then the required formula follows from
Lemma~\ref{lemma_part4_main0}.

Suppose $\un{\de}\not\in\OmegaDrei(\Q)$ and $(\QT,\un{\de}^{(v)},A_v\,|\,v\in\Ver{\QT})$
is a $\un{\de}$-tree. Define the set
$I=\{v\in\Ver{\QT}\,|\,v\text{ is not a leaf}\,\}$. For $v\in\Ver{\QT}$ denote %
$m_v=m(\Q_{\un{\de}^{(v)}})\leq m$, $n_v=\#\Ver{\Q_{\un{\de}^{(v)}}}$, and
$d_v=\#\Ver{\Q_{\un{\de}^{(v)}}}$. If $v\in\Ver{\QT}\backslash I$, then
$\un{\de}^{(v)}\in\OmegaDrei(\Q)$ and Lemma~\ref{lemma_part4_main0} together with the
inequalities $m_v\leq m\leq n$ and $n_v\leq d_v$  implies
\begin{eq}\label{eq1_theo_part4}
|\un{\de}^{(v)}|\leq m_v(d_v-n_v)+m_v\leq m(d_v-n_v)+n_v.
\end{eq}
For $v\in I$ let $A_v=(a_{v1},\ldots,a_{vt_v})$. We define
$\un{\theta}^{(v)}=\un{\de}^{(v)}-2\sum_{i=1}^{t_v}\mdeg(a_{vi})$ and
$k_v=n_v-\#\Ver{\Q_{\un{\theta}^{(v)}}}$.
By~\Ref{eq1_theo_part4}, we can apply Lemma~\ref{lemma_pA26_10} to all vertices of $I$
starting from elements of the set $\{v\in I\,|\,a'\text{ is a leaf for every
}a\in\Arr{\QT}\text{ with }a''=v\}$. Hence we obtain
$$|\un{\de}|\leq m(d-n)+n+\rho,$$
where $\rho=\sum_{v\in I}\rho_v$ and $\rho_v=2\sum_{i=1}^{t_v}\deg(a_{vi}) -
m(t_v+1)-k_v$.

We consider closed paths $c_1,\ldots,c_{l_1}$ from Lemma~\ref{lemma_pA26_8}, where
$l_1\leq\#I-1$. For every $v\in I$ we define $J_v\subset[1,t_v]$ by the equality
$C_v=A_v\backslash\{c_1,\ldots,c_{l_1}\}=\{a_{vi}\}_{i\in J_v}$ and denote $I_0=\{v\in
I\,|\,C_v\neq\emptyset\}$. Therefore,
$$\rho=\left(2\sum_{v\in I_0}\sum_{i\in J_v} \deg(a_{vi}) - m(t + \#I) - \sum_{v\in I} k_v\right) + %
2\sum_{i=1}^{l_1}\deg(c_i),
$$
where $t$ stands for $\sum_{v\in I}t_v=l_1 + \sum_{v\in I_0}\#J_v$. Since $\deg(c_i)\leq
m$ and $l_1-\#I\leq -1$, we have
$$\rho\leq \sum_{v\in I_0}\left(2\sum_{i\in J_v}\deg(a_{vi}) - m\#J_v - k_v\right) - m.$$
For all $v\in I_0$ define $n_{v0}$ for the chain of paths $C_v$ in the same way as we
have defined $n_0$ in Lemma~\ref{lemma_pA26_9} and let $s_v$ be the number of vertices in
$C_v$. Lemma~\ref{lemma_pA26_9} together with the inequality $-k_v\leq -n_{v0}$ implies %
$\rho\leq \sum_{v\in I_0}s_v - m$. Since there is no $u\in\Ver{\Q}$ that belongs to
$C_{v_1}$ and $C_{v_2}$ for different $v_1,v_2\in I_0$, we have $\sum_{v\in I_0}s_v\leq
n$ and $\rho\leq n-m$.
\end{proof}

\section{Properties of a closed path $h$ with $h\not\equiv0$}\label{section_part2}

In this section $\Q$ is a strongly connected quiver and $\Char(K)=2$. Let $a=a_1\cdots a_s$ be a primitive closed path in $\Q$ and $v_1=a_1'',\ldots,v_s=a_s''$, where $a_1,\ldots,a_s\in\Arr{\Q}$ and $s\geq2$. Suppose $h$ is a closed path in $\Q$.

\bigskip
\noindent{\bf Definition (of good subpaths).} A subpath $b$ in $h$ is called {\it good},
if
\begin{enumerate}
\item[a)] $b',b''\in\{v_1,\ldots,v_s\}$;

\item[b)] $\degII{v_i}{b}=0$ for all $i$;

\item[c)] $b\neq a_i$ for all $i$.
\end{enumerate}
Suppose $h\sim b_1g_1b_2g_2$, where $b_1,b_2$ are good subpaths in $h$ and $g_1,g_2$ are
paths in $\Q$. Then we say that $b_1$ and $b_2$ are {\it different} subpaths in $h$.

If we change part~$c)$ of the definition of a good path into
\begin{enumerate}
\item[c')] $b\neq a_i$ for every $i$ satisfying $\deg_{a_i}(h)\leq2$,
\end{enumerate}
then we obtain the definition of a {\it semi-good} subpath $b$ in $h$.

\bigskip
\noindent{\bf Definition (of good components).} A subset $I\subset\{v_1,\ldots,v_s\}$ is
called a {\it good component} with respect to $h$, if the following conditions are valid:
\begin{enumerate}
\item[a)] For every good subpath $b$ in $h$ we have $b'\in I$ if and only if $b''\in I$.

\item[b)] There is a good subpath $b$ in $h$ such that $b'\in I$.

\item[c)] The set $I$ is a minimal (by inclusion) subset of $\{v_1,\ldots,v_s\}$ that
satisfies~$a)$ and~$b)$.
\end{enumerate}
Taking semi-good subpaths instead of good subpaths in the above definition, we obtain the
definition of a {\it semi-good component}.

Let $I_1,\ldots,I_r$ be all good components with respect to $h$. Obviously,
\begin{eq}\label{eq_good_comp_dec}
\{v_1,\ldots,v_s\}=I_0\sqcup I_1\sqcup \cdots \sqcup I_r
\end{eq}
for some $I_0\subset \{v_1,\ldots,v_s\}$. Formula~\Ref{eq_good_comp_dec} is called the
{\it decomposition into good components} with respect to $h$ and $I_0$ is called the {\it
null component} with respect to $h$.

\smallskip
In what follows we assume that $\deg_{a_i}(h)=2$ for all $i$ unless otherwise stated.

\begin{lemma}\label{lemma_good_components}
1. For every good subpath $b$ in $h$ we have $b'\not\in I_0$ and $b''\not\in I_0$.

2. For all $u,w\in I_j$, where $j>0$, there are pairwise different good subpaths
$b_1,\ldots,b_l$ in $h$ such that $b_1\cdots b_l$ is a closed path in $\Q$, $b_1''=u$,
and $b_k'=w$ for some $k$ with $1\leq k\leq l$.
\end{lemma}
\begin{proof}
Part~1 follows from the definition. Let $\QG$ be the $h$-restriction of $\Q$ to the
vertices $v_1,\ldots,v_s$ (see Section~\ref{subsection_notations} for the definition). We consider $h$ as a path in $\QG$ and define
$\un{\theta}\in\NN^{\Arr{\QG}}$ as $\un{\theta}=\mdeg(h)-2\mdeg(a)$. Since
$\deg_{a_i}(h)=2$ for all $i$, it is not difficult to see that for the decomposition
$\un{\theta}=\un{\theta}^{(1)}+\cdots+\un{\theta}^{(r)}$ with respect to
$\Q_{\un{\theta}}$ we have $\Ver{\Q_{\un{\theta}^{(j)}}}=I_j$ for any $1\leq j\leq r$.
To conclude the proof, we apply Lemma~\ref{lemma_mdeg} to $\un{\theta}^{(j)}$ and $\Q_{\un{\theta}^{(j)}}$.
\end{proof}

\begin{lemma}\label{lemma_p14}
If $h\not\equiv a^2\!f$ for all $f\in\path(\Q)$, then the number of good components with
respect to $h$ is equal or greater than two.
\end{lemma}
\begin{proof}
Let $r$ be the number of good components and $\un{\de}=\mdeg(h)$. If $r=0$, then
$$
\de_b= \left\{
\begin{array}{rl}
0,&\text{if } b\neq a_i \text{ for all } i\\
2,&\text{otherwise }\\
\end{array}
\right.
$$
for $b\in\Arr{\Q}$. Hence $h\sim a^2$  and we have a contradiction.

Suppose $r=1$. If $v_i\in I_0$, then $h\sim a_{i-1}a_if_1 a_{i-1}a_if_2$ for some paths
$f_1,f_2$ that do not contain $a_{i-1}$ and $a_i$. Substitute a new arrow $a_{s+1}$ for
the path $a_{i-1}a_i$. Repeat this procedure for all elements of $I_0$. Thus we can
assume that $I_0=\emptyset$ and $I=\{v_1,\ldots,v_s\}$ is the only good component.

If $s=1$, then Lemma~\ref{lemma_L0} implies a contradiction. Otherwise, we consider the
$h$-restriction of $\Q$ to $v_1,\ldots,v_s$, remove arrows $a_1,\ldots,a_s$ from this
restriction, and denote the resulting quiver by $\QG$. Let $\QT$ be a spanning tree for
$\QG$, i.e.,
\begin{enumerate}
\item[a)] $\Ver{\QT}=\{v_1,\ldots,v_s\}$ and $\Arr{\QT}\subset\Arr{\QG}$;

\item[b)] If we consider $\QT$ as a graph without orientation, then it is a tree.
\end{enumerate}
Consider a leaf $v_i$ of $\QT$ together with the unique arrow $b\in\Arr{\QT}$ satisfying
$v_i\in\{b',b''\}$. Then the condition of Lemma~\ref{lemma_L4} is true and we have
$h\equiv a_{i-1}a_i f_1\, a_{i-1}a_i f_2$ for some $f_1,f_2\in\path(\Q)$. We remove the
vertex $v_i$ and the arrow $b$ from $\QT$ and denote the resulting quiver by $\QT_1$. As
above, we consider some leaf of $\QT_1$, apply Lemma~\ref{lemma_L4}, and so on. Finally,
we obtain $h\equiv af_1af_2\equiv a^2f_1f_2$ for some paths $f_1,f_2\in\path(\Q)$; a
contradiction.
\end{proof}

\begin{lemma}\label{lemma_predl_p16}
Suppose $\{v_1,\ldots,v_s\}=I_0\sqcup I_1\sqcup\cdots\sqcup I_r$ is the decomposition
into good components with respect to $h$, $r\geq2$, and
$V\subset\{v_1,\ldots,v_s\}\backslash I_1$. Let $b,c,e$ be pairwise different
good subpaths in $h$ such that
\begin{enumerate}
\item[a)] $b'\in I_1$ and $c',e'\in V$;

\item[b)] $v\in\Ver{b}\cap\Ver{c}\cap\Ver{e}$ for some $v$.
\end{enumerate}
Then there exists a closed path $h_0$ in $\Q$ such that $h_0\equiv h$ and
$$\{v_1,\ldots,v_s\}=I_0\sqcup\ov{I}_1\sqcup \bigsqcup_{k\in D} I_k\sqcup J_1\sqcup\cdots\sqcup J_l$$
is the decomposition into good components with respect to $h_0$, where $l\geq0$ and
$D=\{2,\ldots,r\}\backslash\{i,j\}$ for $c'\in I_i$, $e'\in I_j$. Moreover,
$\#\ov{I}_1>\#I_1$.
\end{lemma}
\begin{proof}
We have $b=b_1b_2$, $c=c_1c_2$, and $e=e_1e_2$ for some paths $b_i,c_i,e_i$ in $\Q$
($i=1,2$) with $b_1'=c_1'=e_1'=v$. There are two possibilities:

\textbf{1.} If $h\sim b_2 f_1 c_1 \cdot c_2 f_2 e_1 \cdot e_2 f_3 b_1$ for some
$f_1,f_2,f_3\in\path(\Q)$, then we define $h_0=b_2 f_1 c_1 \cdot e_2 f_3 b_1 \cdot c_2
f_2 e_1$ and we have $h_0\equiv h$. Let $S_0$ and $S$ be the sets of good subpaths in
$h_0$ and $h$, respectively. Then $S_0=(S\cup\{b_1c_2,c_1e_2,e_1b_2\})\backslash
\{b,c,e\}$. Clearly, $I_k$ is a good component with respect to $h_0$, where $2\leq k\leq
r$ and $k\neq i,j$, and $I_0$ is the null component with respect to $h_0$. By part~2 of
Lemma~\ref{lemma_good_components}, the set $I_1$ and the vertices $c'$ and $e''$ belong
to one and the same good component with respect to $h_0$, which we denote by $\ov{I}_1$.
Thus $\#\ov{I}_1>\# I_1$ and the claim is proven.

\textbf{2.} If $h\sim b_2 f_1 e_1 \cdot e_2 f_2 c_1 \cdot c_2 f_3 b_1$ for some
$f_1,f_2,f_3\in\path(\Q)$, then the proof is analogous.
\end{proof}

\begin{lemma}\label{lemma_partII_main}
Suppose $h\not\equiv a^2\!f$ for all $f\in\path(\Q)$. Then there exists a closed path
$h_0$ in $\Q$ and a good component $I$ with respect to $h_0$ such that $h_0\equiv h$ and
if good subpaths $b,c$ in $h_0$ and $v\in\Ver{\Q}$ satisfy the following condition:
\begin{eq}\label{eq_lemma_partII_main}
b'\in I,\, c'\not\in I,\text{ and }v\in \Ver{b}\cap\Ver{c},
\end{eq}
then
\begin{enumerate}
\item[a)] $b$ is the unique good subpath in $h_0$ satisfying~\Ref{eq_lemma_partII_main},
i.e., $h_0\sim bf_0$, where $f_0$ do not contain a good subpath $b_1$ with $b_1'\in I$
and $v\in\Ver{b_1}$;

\item[b)] $\deg_v(b)=1$.
\end{enumerate}
\end{lemma}
\begin{proof} The proof consists of two parts. At first we find $h_0$ and $I$ that satisfy
condition~a), then we change $h_0$ to make condition~b) valid.

\textbf{a)} For a good component $I$ with respect to $h$ and
$V\subset\{v_1,\ldots,v_s\}\backslash I$, we write $I>V$ if the condition of
Lemma~\ref{lemma_predl_p16} does not hold for $I_1=I$ and $V$.

Suppose $\{v_1,\ldots,v_s\}=I_0\sqcup I_1\sqcup \cdots \sqcup I_r$ is the decomposition
into good components with respect to $h$. If $I_1\not>I_2\sqcup\cdots\sqcup I_r$, then
Lemma~\ref{lemma_predl_p16} implies that there is an $h^{(0)}\equiv h$ such that
$\#\ov{I}_1>\#I_1$ for a good component $\ov{I}_1$ with respect to $h^{(0)}$. Repeat this
procedure for $\ov{I}_1$ and so on. Finally, we obtain $h_1\equiv h$ such that
$I_{11},\ldots,I_{1r_1}$ are all good components with respect to $h_1$ and
$I_{11}>I_{12}\sqcup\cdots\sqcup I_{1r_1}$. Note that $r_1\geq2$ by
Lemma~\ref{lemma_p14}.

If $I_{12}\not>I_{13}\sqcup\cdots\sqcup I_{1r_1}$, then we act as above; and so on.
Finally, we obtain $h_l\equiv h$ such that $I_{l1},\ldots,I_{lr_l}$ are all good
components with respect to $h_l$ and $I_{ii}>I_{i,i+1}\sqcup\cdots\sqcup I_{ir_l}$ for
any $1\leq i< r_l$. Then condition~a) holds for $h_0=h_l$ and $I=I_{lr_l}$.

\textbf{b)} Consider $h_0$ and $I$ that have been constructed in part~a) of the proof.
Suppose $b,c$ are good subpaths with respect to $h_0$, $b'\in I$, $c'\not\in I$, and
$v\in\Ver{\Q}$. If $\deg_v(b)\geq2$, then $b=b_1qb_2$ for some paths $b_1,q,b_2$
satisfying $q'=q''=v$ and $\deg_v(b_1b_2)=0$. Assume that $c=c_1c_2$ for paths $c_1,c_2$
with $c_1'=c_2''=v$ and $h\sim bf_1c f_2$ for some paths $f_1,f_2$. Then %
$$h_0\sim b_2 f_1 c_1\cdot c_2 f_2 b_1\cdot q\equiv  c_2 f_2 b_1\cdot  b_2 f_1 c_1\cdot q,$$
and we define $h_1=c_2 f_2 b_1\cdot  b_2 f_1 c_1\cdot q$.  Let $S_1$ and $S_0$ be the
sets of good subpaths in $h_1$ and $h_0$, respectively. Then
$S_1=(S_0\cup\{b_1b_2,c_1qc_2\})\backslash \{b,c\}$. It is not difficult to see that
every good component with respect to $h_1$ is a good component with respect to $h_0$ and
vice versa. Moreover, condition~a) remains valid for $h_1$.

If condition~b) of the lemma does not hold for $h_1$ and some paths $b$ and $c$, then we
repeat the above procedure for $h_1$; and so on. Denote by $k$ the sum $\sum\deg{b}$ that
ranges over all $b\in\Arr{\Q}$ with $b'\in I$. After each step of the procedure $k$ is
diminished by a positive number. Hence we finally obtain $h_0$ that satisfies
conditions~a) and~b).
\end{proof}

Now we assume that $h$ is a closed path in $\Q$ with $\deg_{a_i}(h)\geq2$ for all $i$.

\begin{lemma}\label{lemma_partII_main2}
Suppose $h\not\equiv a^2\!f$ for all $f\in\path(\Q)$. Then there exists a closed path
$h_0$ in $\Q$ and a semi-good component $I$ with respect to $h_0$ such that $h_0\equiv h$
and if good subpaths $b,c$ in $h_0$ and $v\in\Ver{\Q}$ satisfy~\Ref{eq_lemma_partII_main}
then conditions~a) and~b) of Lemma~\ref{lemma_partII_main} are valid.
\end{lemma}
\begin{proof}
Suppose $h\sim a_ic_1\cdots a_ic_l$ for some $1\leq i\leq s$, where
$l=\deg_{a_i}(h)\geq3$. Then we add a new arrow $b_i$ to $\Q$ and define $b_i'=a_i'$,
$b_i''=a_i''$. Moreover, we substitute $a_ic_1a_ic_2b_ic_3\cdots b_ic_l$ for $h$. After
performing this procedure for all $i$ we obtain a strongly connected quiver $\QG$ and a
closed path $h_1$ in $\QG$ satisfying $\deg_{a_i}(h_i)=2$ for all $i$.
Lemma~\ref{lemma_partII_main} completes the proof.
\end{proof}

\section{The main upper bound}\label{section_part3}

Suppose $\Q\in\Q(n,d,m)$ is a quiver and $\Char(K)=2$. The set $\OmegaZwei(\Q)$ has been defined
in Section~\ref{section_part4}. This section is dedicated to the proof of the following
theorem.

\begin{theo}\label{theo_part3}
If $h\not\equiv0$ is a closed path in $\Q$, then $\deg(h)\leq m(d-n-1)+2n$.
\end{theo}

\begin{lemma}\label{lemma_p26}
Suppose $h$ is a closed path in $\Q$ such that $h\not\equiv0$ and
$\mdeg(h)\not\in\OmegaZwei(\Q)$. Then $h\equiv cf$, where $f$ and $c=c_1b_1c_2b_2c_3$ are
closed paths in $\Q$, paths $c_1,c_2,c_3$ can be empty, $b_1,b_2\in\Arr{\Q}$, and the
following conditions hold:
\begin{enumerate}
\item[a)] $\deg_{b_1}(h)=\deg_{b_2}(h)=1$;

\item[b)] $\Ver{c_1}\cap \Ver{c_2}=\emptyset$, $\Ver{c_2}\cap \Ver{c_3}=\emptyset$, and %
$\Ver{c_1}\cap \Ver{c_3}=c''_1=c'_3=v_0$. Schematically, this condition is depicted as%
$$
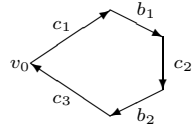
\begin{picture}(100,30)
\put(20,15){%
\put(0,0){\vector(3,2){30}}\put(8,13){$\scriptstyle c_1$}%
\put(30,-20){\vector(-3,2){30}}\put(8,-15){$\scriptstyle c_3$}%
\put(30,20){\vector(2,-1){20}}\put(40,18){$\scriptstyle b_1$}%
\put(50,-10){\vector(-2,-1){20}}\put(40,-22){$\scriptstyle b_2$}%
\put(50,10){\vector(0,-1){20}}\put(54,-2){$\scriptstyle c_2$}%
\put(-8,-2){$\scriptstyle v_0$}%
}%
\end{picture}
$$

\item[c)] for all $v\in\Ver{c}$ with $\deg_v(c)\geq2$ we have $\deg_v(c)=\deg_v(h)$. In
particular, $\deg_{v_0}(c)=1$;

\item[d)] for all $v\in\Ver{c_2}$ we have $\deg_v(h)>\deg_v(c)=1$.
\end{enumerate}
\end{lemma}
\begin{proof}
Since $\mdeg(h)\not\in\OmegaZwei(\Q)$, there is a $\mdeg(h)$-double path $a$ in $\Q$ such that %
$\mdeg(\Q)-2\mdeg(a)$ is indecomposable. 
We apply Lemma~\ref{lemma_partII_main2} to $h$ and $a$ to obtain a closed path $h_0$ in
$\Q$ and a semi-good component $I$ satisfying the conditions from
Lemma~\ref{lemma_partII_main2}. Without loss of generality we can assume that $h=h_0$. In
what follows, all good subpaths in $h$ will be considered with respect to $a$. Define the
subset
$$V\subset\Ver{\Q}\backslash\Ver{a}$$ %
that contains $v$ if and only if there is a good subpath $e$ in $h$
such that $e'\not\in I$ and $v\in\Ver{e}$. %
Since $\mdeg(\Q)-2\mdeg(a)$ is indecomposable, there is a good subpath $b=b_1\cdots b_l$
in $h$ satisfying $b'\in I$ and $\Ver{b}\cap V\neq\emptyset$, where
$b_1,\ldots,b_l\in\Arr{\Q}$. Since $b''_1\not\in V$ and $b'_l\not\in V$, we can define
$$i=\min\{1\leq k\leq l\,|\,b'_k\in V\}\text{ and }%
j=\min\{i< k\leq l\,|\,b'_k\not\in V\}.
$$
Let $\deg_{b_i}(h)>1$, then there is a good subpath $e$ in $h$ with $b_i\in\Arr{e}$ and
$h\sim b f_1 e f_2$ for some $f_1,f_2\in\path(\Q)$ because $\deg_{b'_i}(b)=1$ by
Lemma~\ref{lemma_partII_main2}. If $e'\in I$, then we obtain a contradiction to
Lemma~\ref{lemma_partII_main2}. If $e'\not\in I$, then $b''_i\in V$; a contradiction.
Therefore, $\deg_{b_i}(h)=1$ and, similarly, $\deg_{b_j}(h)=1$.

If $\Ver{b_1\cdots b_{i-1}}\cap\Ver{b_{j+1}\cdots b_{l}}\neq\emptyset$, then $h\sim c f$
for $c=c_1b_ic_2b_jc_3$ satisfying part~b) of the lemma, where $f$ is a path in $\Q$,
$c_2=b_{i+1}\cdots b_{j-1}\in\path(\Q)$, and $c_1,c_3\in\path(\Q)$ are subpaths in
$b_1\cdots b_{i-1}$, $b_{j+1}\cdots b_{l}$, respectively. Moreover, we can assume that
$\deg_{v_0}(c_1)=\deg_{v_0}(c_3)=1$.

If $\Ver{b_1\cdots b_{i-1}}\cap\Ver{b_{j+1}\cdots b_{l}}=\emptyset$, then, taking into
account Lemma~\ref{lemma_aaa}, we have $h\equiv c f$, where $c$ satisfies the same
conditions as above and $f$ is a path in $\Q$.

If there is a $v\in\Ver{c_1}$ such that $\deg_v(h)>\deg_v(c_1)\geq2$, then
$c_1=e_1e_2e_3$, where $e_1,e_2,e_3\in\path(\Q)$, $e_2'=e_2''=v$, and
$\deg_v(e_1)=\deg_v(e_3)=1$. We have $h\sim e_1\cdot e_2\cdot e_3b_ic_2b_jc_3 f_1 \cdot
p\cdot f_2$ for some
$f_1,f_2\in\path(\Q)$ and $p'=p''=v$. Thus %
$h\equiv e_1\cdot e_3b_ic_2b_jc_3 f_1\cdot  e_2\cdot p\cdot f_2$ and we change notations
by putting $c=e_1e_3b_ic_2b_jc_3$. We repeat this procedure for all vertices of
$c_1,c_2,c_3$ and obtain that part~c) of the lemma holds.

Since $\Ver{c_2}\subset V$, part~d) of the lemma is a consequence of part~c).
\end{proof}

We assume that $h\equiv cf\not\equiv0$ is a closed path from Lemma~\ref{lemma_p26}, where
$c=c_1b_1c_2b_2c_3$. Define sets
$$S_{arr}=\{a\in\Arr{c_3c_1}\,|\,\deg_a(h)=\deg_a(c)\},$$
$$S_{ver}=\{v\in\Ver{c_3c_1}\,|\,\deg_v(h)=\deg_v(c)\}.$$
In this section we will use the next remark.

\begin{remark}
\begin{enumerate}
\item[1.] Since $f$ is not an empty path, we have $v_0\not\in S_{ver}$.

\item[2.] For all $v\in\Ver{c_i}$, $a\in\Arr{c_i}$ the equalities
$\deg_v(c)=\deg_v(c_i)$, $\deg_a(c)=\deg_a(c_i)$ hold ($i=1,2,3$).
\end{enumerate}
\end{remark}

\begin{lemma}\label{lemma_p27}
For all $v\in\Ver{c_3c_1}$, $a\in\Arr{c_3c_1}$ we have
\begin{enumerate}
\item[a)] if $a'\in S_{ver}$ or $a''\in S_{ver}$, then $a\in S_{arr}$;

\item[b)] if $a'\not\in S_{ver}$ or $a''\not\in S_{ver}$, then $\deg_a(c)=1$;

\item[c)] if $v\not\in S_{ver}$, then $\deg_v(c)=1$.
\end{enumerate}
\end{lemma}
\begin{proof} Part~a) is trivial. Part~c) follows from part~c) of Lemma~\ref{lemma_p26}.
If $a'\not\in S_{ver}$, then $\deg_{a'}(c)=1$ by part~c); hence $\deg_a(c)=1$ and part~b)
is proven.
\end{proof}

The following lemma will help us to perform the induction step in the proof of
Theorem~\ref{theo_part3}.

\begin{lemma}\label{lemma_p28}
We assume that for every strongly connected quiver $\QG$ with $\#\Arr{\QG}<d$ the
assertion of Theorem~\ref{theo_part3} is valid. Then $\deg(c)\leq
m(\#S_{arr}-\#S_{ver}+1)+2\#S_{ver}$.
\end{lemma}
\begin{proof}
Let $\#S_{ver}=\emptyset$. Then $c$ is a primitive closed path by parts~c) and~d) of
Lemma~\ref{lemma_p26}. Hence $\deg(c)\leq m$.

Let $\#S_{ver}\geq1$. Using the fact that $v_0\not\in S_{ver}$ and part~b) of
Lemma~\ref{lemma_p26} we obtain
\begin{eq}\label{eq_p28_star1}
S_{ver}=S_{ver}^1\sqcup S_{ver}^3 \text{ and } S_{arr}=S_{arr}^1\sqcup S_{arr}^3
\end{eq}
for $S_{arr}^1=S_{arr}\cap\Arr{c_1}$, $S_{arr}^3=S_{arr}\cap\Arr{c_3}$,
$S_{ver}^1=S_{ver}\cap\Ver{c_1}$, and $S_{ver}^3=S_{ver}\cap\Ver{c_3}$.

Suppose $\#S_{ver}^1\geq1$. We assume that $c_1=x_1\cdots x_s$ for
$x_1,\ldots,x_s\in\Arr{\Q}$, where for $i\neq j$ we can have $x_i=x_j$. We define
$p=\min\{1\leq k\leq s\,|\,x_k'\in S_{ver}^1\}$ and $q=\max\{1\leq k\leq s\,|\,x_k'\in
S_{ver}^1\}$. Then $c_1=e_1e_2e_3$, where paths $e_1=x_1\cdots x_p$, $e_2=x_{p+1}\cdots
x_q$, and $e_3=x_{q+1}\cdots x_s$ can be empty. We claim that
\begin{eq}\label{eq_p28_star2}
\deg(e_2)\leq m(\#S_{arr}^1-\#S_{ver}^1)+2\#S_{ver}^1.
\end{eq}
To prove the claim we consider the $e_2$-restriction of $\Q$ to $S_{ver}^1$, add a new
arrow $z$ from $e_2'$ to $e_2''$, and denote the resulting quiver by $\QG$.
In other words, $\Ver{\QG}=S_{ver}^1$ and $a\in\Arr{\QG}$ has one of the
following types:
\begin{enumerate}
\item[1.] $a=\widetilde{x_i}$, where $1\leq i\leq s$ and $x'_i,x''_i\in S_{ver}^1$;

\item[2.] $a=\widetilde{x_i\cdots x_j}$ for $1\leq i<j\leq s$, where $x''_i,x'_j\in
S_{ver}$ and $x'_i,\ldots,x'_{j-1}\not\in S_{ver}$;

\item[3.] $a=z$.
\end{enumerate}
Note that for an arrow $a=\widetilde{x_i}$ of type 1 we have $x_i\in S_{arr}$ by part~a)
of Lemma~\ref{lemma_p27} and we say that $x_i$ is assigned to $a$. Similarly, for an
arrow $a=\widetilde{x_i\cdots x_j}$ of type 2 we have $x_i,x_j\in S_{arr}$ and we say
that $x_i,x_j$ are assigned to $a$; moreover,
\begin{enumerate}
\item[a)] $\deg_{x_k}(e_2)=1$ for any $i\leq k\leq j$ (see part~b) of
Lemma~\ref{lemma_p27}).

\item[b)] $\deg_{x'_k}(c)=\deg_{x'_k}(e_2)=1$ for any $i\leq k\leq j-1$ (see part~c) of
Lemma~\ref{lemma_p27}). In particular, $x_i\cdots x_j$ is either a primitive closed path
or it is a subpath of $c$ without self-intersections; thus, $\deg(x_i\cdots x_j)\leq m$.
\end{enumerate}

Let $y$ be the unique path in $\QG$ that corresponds to the path $e_2$ in $\Q$. The
quiver $\QG$ is strongly connected, since $yz$ is a closed path in $\QG$ that contains
all arrows and all vertices of $\QG$. Moreover, we have $yz\not\equiv0$, since
$\deg_a(y)=1$ for every arrow $a$ of type 2, $\deg_z(y)=0$, and $h\not\equiv0$.

For every arrow $a$ of type 1 there is an arrow from $S_{arr}^1$ that is assigned to $a$;
and for every arrow $b$ of type 2 there are two arrows from $S_{arr}^1$ that are assigned
to $b$. But the arrow $x_p\in S_{arr}^1$ is not assigned to any arrow of $\QG$. Therefore,
$$\#\Arr{\QG}-1\leq \#S_{arr}^1-l-1,$$
where $l$ stands for the number of arrows of type 2. Since $b_1,b_2\not\in S_{arr}^1$,
it follows that %
$\#\Arr{\QG}\leq \#S_{arr}^1 <\Arr{\Q}=d$. Applying Theorem~\ref{theo_part3} to $\QG$, we
obtain
$$\deg(yz)\leq m(\QG)(\#\Arr{\QG}-\#\Ver{\QG})+2\#\Ver{\QG}.$$
It is not difficult to see that $m(\QG)\leq m$. Thus
$$\deg(y)\leq m(\#S_{arr}^1-\#S_{ver}^1-l)+2\#S_{ver}^1 - 1.$$
By property~b) of paths of type 2, we have
$$\deg(e_2)\leq \deg(y)+l(m-1).$$
The last two formulas conclude the proof of~\Ref{eq_p28_star2}.

\smallskip
If $\#S_{ver}^3\geq1$, then we rewrite $c_3$ in a form $c_3=g_1g_2g_3$ in the same way as
we have done for $c_1=e_1e_2e_3$. Then the proof of the formula
\begin{eq}\label{eq_p30_star9}
\deg(g_2)\leq m(\#S_{arr}^3-\#S_{ver}^3)+2\#S_{ver}^3
\end{eq}
is similar to the proof of~\Ref{eq_p28_star2}.

Suppose $S_{ver}^1\neq\emptyset$ and $S_{ver}^3\neq\emptyset$. Then
$$\deg(c)=\deg(c_1b_1c_2b_2c_3)=\deg(e_2)+\deg(g_2)+\deg(f_1)+\deg(f_2),$$
where $f_1=g_3e_1$ and $f_2=e_3b_1c_2b_2g_1$. Parts~c) and~d) of Lemma~\ref{lemma_p26}
imply that
\begin{enumerate}
\item[a)] for every $v\in (\Ver{f_1}\cup\Ver{f_2})\backslash\{f'_1,f''_1,f'_2,f''_2\}$ we
have $\deg_v(c)=1$;

\item[b)] $(\Ver{f_1}\cup\Ver{f_2}) \cap
(\Ver{e_2}\cup\Ver{g_2})=\{f'_1,f''_1,f'_2,f''_2\}$.
\end{enumerate}
It follows that there are paths $d_1,d_2$ in $\Q$ such that $f_1d_1f_2d_2$ is a primitive
closed path in $\Q$. In particular, $\deg(f_1)+\deg(f_2)\leq m$.
Formulas~\Ref{eq_p28_star2} and~\Ref{eq_p30_star9} conclude the proof of the lemma.

The cases $S_{ver}^1\neq\emptyset,S_{ver}^3=\emptyset$ and
$S_{ver}^1=\emptyset,S_{ver}^3\neq\emptyset$ can be treated in the similar fashion. If
$S_{ver}^1=\emptyset$ and $S_{ver}^3=\emptyset$, then $S_{ver}=\emptyset$; a
contradiction.
\end{proof}

\bigskip
\begin{proof_theo_part3}
Suppose $\Q_{\mdeg(h)}\in\Q(n_0,d_0,m_0)$ for some $n_0,d_0,m_0$. We assume that
the theorem is proven for the case $\Q=\Q_{\mdeg(h)}$. Then we have $\deg(h)\leq
m_0(d_0-n_0-1)+2n_0$. Lemma~\ref{lemma_subquiver} implies
$$\deg(h)\leq m(d_0-n_0)+2n_0-m_0\leq m(d-n-1)+2n-m_0$$
and we obtain the required upper bound. Therefore, without loss of generality we can
assume that $\Q=\Q_{\mdeg(h)}$.

We prove the theorem by induction on $\#\Arr{\Q}$.

\smallskip
\textbf{Induction base.} If $\#\Arr{\Q}=1$, then $\Ver{\Q}=\{v\}$ and the only arrow of
$\Q$ is a loop in $v$. Then $\deg(h)=1$ and the required upper bound on $\deg(h)$ holds.

\smallskip
\textbf{Induction step.} If $\mdeg\in\OmegaZwei(\Q)$, then see Theorem~\ref{theo_part4};
otherwise we apply Lemma~\ref{lemma_p26} to $h$ and obtain $h\equiv cf$, where $f$ and
$c=c_1b_1c_2b_2c_3$ are closed paths in some vertex $v_0$. By Lemma~\ref{lemma_p28}, we
have
\begin{eq}\label{eq_p32_1}
\deg(c)\leq m(\#S_{arr}-\#S_{ver}+1)+2\#S_{ver}.
\end{eq}
We define the quiver $\QG$ by $\Ver{\QG}=\{v\in\Ver{\Q}\,|\,\deg_v(h)>\deg_v(c)\}$ and
$\Arr{\QG}=\{a\in\Arr{\Q}\,|\,\deg_a(h)>\deg_a(c)\}\bigcup \{x\}$, where $x$ is a new
loop in the vertex $v_0$. Then $xf$ is a closed path in $\QG$ that contains all vertices
and arrows of $\QG$. In particular, $\QG$ is a strongly connected quiver and
$\QG=\QG_{\mdeg(xf)}$. Since $cf\not\equiv0$, we have $xf\not\equiv0$. By parts~a) and~d)
of Lemma~\ref{lemma_p26}, we have $\#\Arr{\QG}\leq d-\#S_{arr}-1$ and
$\Ver{\QG}=n-\#S_{ver}$. Applying induction hypothesis to the closed path $xf$ in $\QG$
and using the inequalities $\Arr{\QG}>\Ver{\QG}$ and $m(\QG)\leq m$, we obtain
$$\deg(xf)\leq m(d-n-1)+2n-m(\#S_{arr}-\#S_{ver}+1)-2\#S_{ver}.$$
Formula~\Ref{eq_p32_1} implies the required upper bound on $\deg(h)$.
\end{proof_theo_part3}

\section{The upper bound for the case of small $d$}\label{section_part6}

Assume that $\Char(K)=2$. The following lemma is a stronger version of Lemma~\ref{lemma_pA291}.

\begin{lemma}\label{lemma_pA292_pA293}
Suppose $\Q\in\Q(n,d,m)$. Then using the notation of Lemma~\ref{lemma_pA291} we
have $r\geq1$.
\end{lemma}
\begin{proof} Suppose $r=0$. %
Then $\mdeg(h)=2\sum_{j=1}^t\mdeg(c_j)$, where $t\geq1$, and we have two possibilities.

\textbf{1.} Let $\mdeg(h)\not\in\OmegaZwei(\Q)$. Then there exists a primitive closed
path $a=a_1\cdots a_s$ in $\Q$ ($a_1,\ldots,a_s\in\Arr{\Q}$) such that
$\mdeg(h)-2\mdeg(a)$ is indecomposable and $\deg_{a_i}(h)\geq2$ for all $i$. It is not
difficult to see that $\Ver{a}=I\sqcup J$, where
\begin{enumerate}
\item[1)] $\deg_v(h)=2$ for all $v\in I$;

\item[2)] for every $u,v\in J$ with $u\neq v$ there is a path $g$ in $\Q$ from $u$ to
$v$; moreover, for every $e\in \Arr{g}$ we have $\deg_{e}(h)\geq2$, if $e\not\in
\Arr{a}$; and $\deg_{e}(h)\geq4$, if $e\in \Arr{a}$. Lemma~\ref{lemma_aaa} implies that
$h\equiv gf$ for some path $f$.
\end{enumerate}
If $s>1$, then, applying Lemma~\ref{lemma_L4}, we have $h\equiv a_1a_2f_1a_1a_2f_2\equiv
a_1a_2a_3f_3a_1a_2a_3f_4\equiv\cdots\equiv a f_{2s-3} a f_{2s-2}$ for some paths
$f_1,\ldots,f_{2s-2}$. Lemma~\ref{lemma_L0} gives $h\equiv0$ for $s\geq1$; a
contradiction.

\textbf{2.} If $\mdeg(h)\in\OmegaZwei(\Q)$, then we consider a $\mdeg(h)$-tree
$(\QT,\un{\de}^{(v)},A_v\,|\,v\in\Ver{\QT})$ constructed in Section~\ref{section_part4}.
For a leaf $v\in\Ver{\QT}$ we have $\un{\de}^{(v)}\in\OmegaDrei(\Q_{\un{\de}^{(v)}})$; a
contradiction.
\end{proof}

\begin{theo}\label{theo_part6}
Suppose $\Q\in\Q(n,d,m)$, $h$ is a closed path in $\Q$, and $h\not\equiv0$. Then
$\deg(h)\leq 2m(d-n)+m$.
\end{theo}
\begin{proof}
Using the notation of Lemma~\ref{lemma_pA291} we have $\deg{h}\leq m(r+2t)$ and $r+t\leq
d-n+1$. Lemma~\ref{lemma_pA292_pA293} implies
$$r+2t\leq 2r-1+2t\leq 2(d-n)+1$$
and we obtain the required upper bound.
\end{proof}

\section{Examples}\label{section_example}

\begin{lemma}\label{lemma_example}
Suppose $\Q(n,d,m)\neq\emptyset$. Then there is a $\Q\in\Q(n,d,m)$ and
a closed path $h$ in $\Q$ such that $h\not\equiv0$ and
\begin{enumerate}
\item[1)] $\deg(h)\geq M(n,d,m)-m$, if $\Char(K)=2$;

\item[2)] $\deg(h)= M(n,d,m)$, if $\Char(K)\neq2$, $d\geq
n+2\left[\frac{n-1}{m}\right]+m$ or $n=m$;
\end{enumerate}
where the definition of $M(n,d,m)$ was given in Section~\ref{section_intro}.
\end{lemma}
\begin{proof} Suppose $\Char(K)=2$.

\smallskip
\textbf{a)} If $m=1$, then $n=1$. For the quiver $\Q$ with one vertex $v$ and loops
$a_1,\ldots,a_d$ in $v$ we have $h=a_1\cdots a_d\not\equiv0$ and $\deg(h)=d$.

\smallskip
\textbf{b)} If $m\geq2$ and $n=m$, then we consider the quiver $\Q\in\Q(n,d,m):$
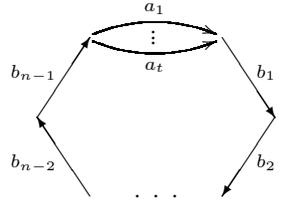
$$
\begin{picture}(0,80)
\put(-45,35){%
\put(0,0){\vector(2,3){20}}%
\put(20,-30){\vector(-2,3){20}}%
\put(70,30){\vector(2,-3){20}}%
\put(90,0){\vector(-2,-3){20}}%
\put(45,-30){%
\put(0,0){\circle*{1}}%
\put(7,0){\circle*{1}}%
\put(-7,0){\circle*{1}}}%
\put(15,30){\xymatrix@C=1.65cm@R=1cm{ %
\ar@/^/@{->}[r]^{a_1} \ar@/_/@{->}[r]_{a_t}&\\
}}%
\put(44,30){\put(0,0){\circle*{1}}\put(0,2){\circle*{1}}\put(0,-2){\circle*{1}}}%
\put(100,-1){,}%
\put(-10,15){$\scriptstyle b_{n-1}$}%
\put(-10,-19){$\scriptstyle b_{n-2}$}%
\put(83,15){$\scriptstyle b_1$}%
\put(83,-19){$\scriptstyle b_2$}%
}%
\end{picture}$$%
where $t=d-n+1\geq1$. For $h=a_1b\cdots a_tb$, where $b=b_1\cdots b_{n-1}$, we have
$\deg(h)=tn$ and $h\not\equiv0$.

\smallskip
\textbf{c)} We assume that $d\geq n+2\left[\frac{n-1}{m}\right]$ and $n>m\geq 2$. Then
$n-1=lm+r$ for $l=\left[\frac{n-1}{m}\right]\geq1$ and $0\leq r\leq m-1$. Consider the
quiver $\Q\in\Q(n,d,m)$:
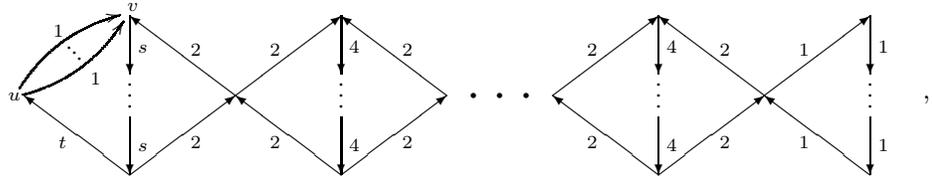
$$
\begin{picture}(0,70)
\put(-160,30){
\put(0,0){\rombL}%
\put(80,0){\rombC}%
\put(200,0){\rombC}%
\put(280,0){\rombR}%
\put(180,0){%
\put(0,0){\circle*{2}}%
\put(10,0){\circle*{2}}%
\put(-10,0){\circle*{2}}}%
\put(340,-1){,}%
}%
\end{picture}
$$
where there are $t=d-n-2l+1\geq1$ arrows from $u$ to $v$, the right primitive closed path
contains $r+1$ arrows, any other primitive closed path contains $m$ arrows, and $s=t+2$. Define
$\un{\de}\in\OmegaNull(\Q)$ in such a way that if a number $k$ is assigned to an arrow
$a\in\Arr{\Q}$, then $\de_a=k$.  Since $\un{\de}\in\OmegaZwei(\Q)$, there is a
closed path $h$ in $\Q$ with $\mdeg(h)=\un{\de}$ and $h\not\equiv0$ by
Remark~\ref{remark_inclusions}. It is not difficult to see that
$\deg(h)=|\un{\de}|=m(d-n-1)+2n-(r+1)$.

\smallskip
\textbf{d)} We assume that $d< n+2\left[\frac{n-1}{m}\right]$ and $n>m\geq 2$. As above,
we have $n-1=lm+r$ for $l\geq1$ and $0\leq r\leq m-1$. Consider the quiver
$\Q\in\Q(n,d,m)$:
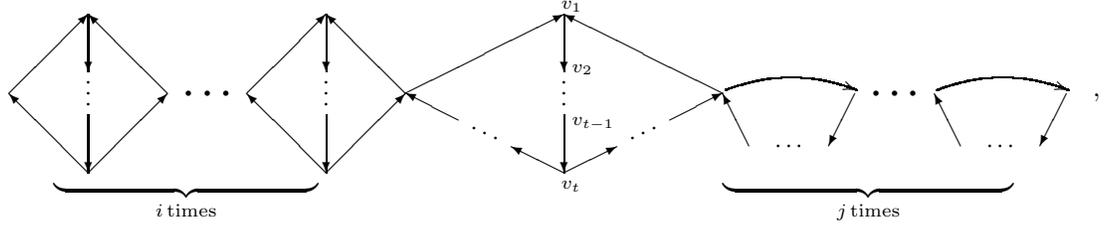
$$
\begin{picture}(0,85)
\put(-210,50){
\put(0,0){\rombSmall}%
\put(90,0){\rombSmall}%
\put(75,0){\put(0,0){\circle*{2}}\put(7,0){\circle*{2}}\put(-7,0){\circle*{2}}}%
\put(150,0){\rombSmallC}%
\put(270,0){\cycleSmall}%
\put(350,0){\cycleSmall}%
\put(335,0){\put(0,0){\circle*{2}}\put(7,0){\circle*{2}}\put(-7,0){\circle*{2}}}%
\put(410,-1){,}%
\put(17,-33){$\underbrace{\qquad\qquad\qquad\qquad\qquad}_{i\,{\rm times}}$}%
\put(270,-33){$\underbrace{\qquad\qquad\qquad\qquad\qquad\quad}_{j\,{\rm times}}$}%
}%
\end{picture}
$$
where every primitive closed path contains $m$ arrows, $i,j\geq0$, $1\leq t< m$, and
$$\begin{array}{ccl}
n&=&m(i+j+2)-j-t,\\
d&=&m(i+j+2)+2i-t+1.\\
\end{array}$$
It is not difficult to see that there exist $i,j,t$ satisfying the given conditions. We define
$\un{\de}\in\OmegaZwei(\Q)$ in a similar way as in part~c). Hence $|\un{\de}|=2m(2i+j+1)$
and $M(n,d,m)-|\un{\de}|=m$.

\smallskip
\textbf{e)} Suppose $\Char(K)\neq2$ and the condition from part~2) of the lemma holds. 

If $m=1$, then we construct the required $h$ similarly to part~a).

If $n=m\geq 2$, then we consider the quiver from part~b). We set $h=a_1 b a_1 b$ if $d\in\{n,n+1\}$ and $h=a_1 b a_2 b a_3 b$ if $d>n+1$. Obviously, $\deg(h)=M(n,d,m)$ and $h\not\equiv0$.
 
Let $n>m\geq2$. We define $l$ and $r$ in the same way as in part~c) and consider the quiver $\Q\in\Q(n,d,m)$: 
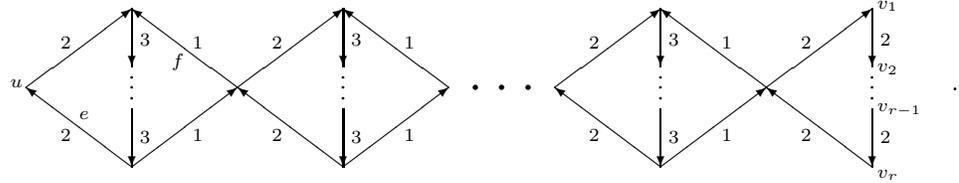
$$
\begin{picture}(0,80)
\put(-160,40){
\put(0,0){\rombCnull}%
\put(80,0){\rombCnull}%
\put(200,0){\rombCnull}%
\put(280,0){\rombRnull}%
\put(180,0){%
\put(0,0){\circle*{2}}%
\put(10,0){\circle*{2}}%
\put(-10,0){\circle*{2}}}%
\put(322,30){$\scriptstyle v_1$}%
\put(322,6){$\scriptstyle v_2$}%
\put(322,-9){$\scriptstyle v_{r-1}$}%
\put(322,-34){$\scriptstyle v_{r}$}%
\put(-6,0){$\scriptstyle u$}%
\put(20,-12){$\scriptstyle e$}%
\put(55,8){$\scriptstyle f$}%
\put(350,-1){.}%
}%
\end{picture}
$$
Here we assume that we have not depicted some loops in $\Q$. Namely, for $s=d-n-2l-r-1\geq0$ there are loops $a$, $b_1,\ldots,b_{s}$ in the vertex $u$ and loops $c_1,\ldots,c_{r}$ in vertices $v_1,\ldots,v_{r}$, respectively.  We assign number $1$ to loops $a$, $c_1,\ldots, c_{r}$ and number $0$ to $b_1,\ldots, b_{s}$. Define
$\un{\de}\in\OmegaNull(\Q)$ in such a way that if a number $k$ is assigned to an arrow
$x\in\Arr{\Q}$, then $\de_x=k$.  Let $h$ be a closed path in $\Q$ with $\mdeg(h)=\un{\de}$. Since $\deg_w(h)=3$ for all $w\in\Ver{\Q}$, we have $\deg(h)=3n$. Lemma~\ref{lemma_indecomposable} (see below) completes the proof.  
\end{proof}

Given a closed path $a=a_1\cdots a_s$ in $\Q$, where $a_i\in\Arr{\Q}$, we write $\tr(X_a)$ for $\tr(X_{a_s}\cdots X_{a_1})$. 

\begin{lemma}\label{lemma_indecomposable} Using notation from part~e) of the proof of Lemma~\ref{lemma_example}, we have $h\not\equiv0$. 
\end{lemma}
\begin{proof} 
Since the construction of $\Q$ and $h$ depend on $l$, we write $\Q_l$ for $\Q$ and $h_l$ for $h$ ($l\geq1$). 

Assume that $h_l\equiv0$. By Lemma~\ref{lemma_reduction}, $\tr(X_{h_l})\equiv0$. Denote $I=\left(
\begin{array}{cc}
1& 0 \\
0& -1\\
\end{array}
\right)$ and 
$J=\left(
\begin{array}{cc}
0& 1 \\
-1& 0\\
\end{array}
\right)$.
We set $X_a=I$, $X_e=J$, and $X_g=E$ for every arrow $g\not\in\{a,e,f\}$ from the left rhombus of $\Q_l$. Since $\tr(I)=\tr(J)=\tr(IJ)=0$, it is not difficult to see that $\tr(X_{h_{l-1}})\equiv0$ in $I(\Q_{l-1},(2,\ldots,2))$, where $h_0$ is defined below. Repeating this procedure, we obtain that $\tr(X_{h_0})\equiv0$ in $I(\Q_0,(2,\ldots,2))$ for 
$$h_0=x_1 y_1 \cdots x_{r+1} y_{r+1}\cdot x_1 \cdots x_{r+1},$$
where $x_1,\ldots,x_{r+1}\in\Arr{\Q_0}$, $x_1\cdots x_{r+1}$ is a closed primitive path in $\Q_0$, $y_i$ is a loop in $x_i'$ ($1\leq i\leq r+1$). For $j=1,2$ we denote 
$$z_{ij}=\left\{
\begin{array}{rl}
y_i,&  \text{if } j=1 \\
1_{x_i'},& \text{otherwise}\\
\end{array}
\right..$$
Since for all $\pi_1,\ldots,\pi_{r+1}\in \Symm_2$ 
$$x_1 z_{1,\pi_1(1)} \cdots x_{r+1} z_{r+1,\pi_{r+1}(1)}\cdot  
x_1 z_{1,\pi_1(2)} \cdots x_{r+1} z_{r+1,\pi_{r+1}(2)} \equiv \sign(\pi_1)\cdots \sign(\pi_{r+1}) h_0,$$
we obtain that $h_0\not\equiv0$. Lemma~\ref{lemma_reduction} implies a contradiction.
\end{proof}


\bigskip
\noindent{\bf Acknowledgements.} This paper was supported by RFFI 10-01-00383a. 


\end{document}